%% file: sdcmultiphys.tex
\documentclass{elsarticle}
\input{preamble.tex}

\usepackage[utf8]{inputenc}
\usepackage[toc,page]{appendix}
\usepackage{amsmath}
\usepackage{amsfonts}
\usepackage{tikz}
\usetikzlibrary{patterns}
\usepackage{subcaption}
\usepackage{pgfplots}
\usepackage[a4paper, total={7in, 9in}]{geometry}
\usepackage{indentfirst}
\usepackage{amsthm}
\usepackage{natbib}
\usepackage{graphicx}
\usepackage{hyperref}
\usepackage{cleveref}
\usepackage{float}
\usepackage{xcolor}
\usepackage{bm}
\usepackage{upgreek}
\usepackage{xifthen}
\usepackage{multicol}
\usepackage{nomencl}
\makenomenclature

\renewcommand{\L}{\mathcal{L}}
\newcommand{\nphys}{{m}}
\newtheorem{remark}{Remark}
\newtheorem{thm}{Theorem}

\def\high{\textit{high}}

\makeatletter
\def\ps@pprintTitle{%
   \let\@oddhead\@empty
   \let\@evenhead\@empty
   \def\@oddfoot{\reset@font\hfil\thepage\hfil}
   \let\@evenfoot\@oddfoot
}
\makeatother

\begin{document}

\begin{abstract}
We present an arbitrarily high-order, conditionally stable, partitioned
spectral deferred correction (SDC) method for solving multiphysics problems
using a sequence of pre-existing single-physics solvers. This method extends
the work in \cite{Huang2019, Huang2019b}, which used implicit-explicit
Runge-Kutta methods~(IMEX) to build high-order, partitioned multiphysics
solvers. We consider a generic multiphysics problem modeled as a system of
coupled ordinary differential equations (ODEs), coupled through coupling
terms that can depend on the state of each subsystem; therefore the method
applies to both a semi-discretized system of partial differential equations
(PDEs) or problems naturally modeled as coupled systems of ODEs.
The sufficient conditions to build arbitrarily high-order partitioned SDC schemes are derived.
Based on these conditions, various of partitioned SDC schemes are designed.
The stability of the first-order partitioned SDC scheme is analyzed in detail on a coupled, linear model problem.
We show that the scheme is conditionally stable, and under conditions on the coupling strength, the scheme
can be unconditionally stable.
We demonstrate the performance of the proposed
partitioned solvers on several classes of multiphysics problems { with moderate coupling strength. They include} a
{stiff} linear system of ODEs, advection-diffusion-reaction systems, and
fluid-structure interaction problems with both incompressible and compressible
flows, where we verify the design order of the SDC schemes and study various
stability properties. We also directly compare the accuracy, stability, and cost
of the proposed partitioned SDC solver with the partitioned IMEX method in
\cite{Huang2019, Huang2019b} on this suite of test problems. The results suggest
that the high-order partitioned SDC solvers are more robust than the partitioned
IMEX solvers for the numerical examples
considered in this work, while the IMEX methods require fewer implicit solves.
\end{abstract}

\begin{frontmatter}

\title{High-order partitioned spectral deferred correction solvers for multiphysics problems}
\author[rvt1]{Daniel Z. Huang}
 \ead{zhengyuh@stanford.edu}
\author[rvt2]{Will Pazner}
\ead{pazner1@llnl.gov}
\author[rvt3]{Per-Olof Persson}
\ead{persson@berkeley.edu}
\author[rvt4]{Matthew J. Zahr}
\ead{mzahr@nd.edu}

\address[rvt1]{Institute for Computational and Mathematical Engineering, Stanford University}
\address[rvt2]{Center for Applied Scientific Computing, Lawrence Livermore National Laboratory}
\address[rvt3]{Department of Mathematics, University of California Berkeley}
\address[rvt4]{Department of Aerospace and Mechanical Engineering, University of Notre Dame}
\end{frontmatter}



\section{Introduction}

The numerical simulation of multiphysics problems involving multiple
physical models or multiple simultaneous physical phenomena is significant
in many engineering and scientific applications, e.g.,
fluid-structure interactions~(FSI) in aeroelasticity~\cite{kamakoti2004fluid, chen2007numerical, huang2018simulation} or biomechanics~\cite{bazilevs2006isogeometric, hron2007fluid, chabannes2013high},
chemical reaction in combustion or subsurface flows~\cite{moin2006large, chen2011multiphysics},
electricity and magnetism with hydrodynamics in plasma physics~\cite{toth2000b, chacon2002implicit, cyr2013new}, among others.
These problems are generally
highly nonlinear, feature multiple scales and strong coupling effects, and
require heterogeneous discretizations for the various physics subsystems.
To balance the treatment of these features, solution strategies ranging from
monolithic approaches to partitioned procedures have been proposed.

In the monolithic approach
\cite{hubner2004monolithic, michler2004monolithic, kuttler2008fixed},
all physical subsystems are solved simultaneously. Therefore, this approach is preferred
in the case of strong interactions to ensure stability. However, when the
coupled subsystems are complex, the monolithic procedure often requires significant implementation effort since only
small components of existing software can be re-used. An alternative is
the partitioned procedure
\cite{farhat2000two, piperno2001partitioned, badia2008fluid}, also known
as a staggered or a loosely coupled procedure, where different subsystems are
modeled and discretized separately, and the resulting equations are solved
independently. The coupling occurs through specific terms that are lagged
to previous time instances and communicated between solvers. This procedure
facilitates software modularity and mathematical modeling; however, these
schemes are often low-order accurate \cite{piperno2001partitioned}
and may suffer from lack of stability \cite{causin2005added}.

Recently, a partitioned solver based on implicit-explicit Runge-Kutta
schemes, first proposed to solve stiff additive ordinary differential
equations \cite{zhong1996additive, ascher1997implicit}, was proposed
\cite{van2007higher, froehle2014high} in the context of a specific
multiphysics system: fluid-structure interaction.
This idea is generalized in~\cite{Huang2019}  to build a framework to construct
high-order, partitioned solvers based on monolithic IMEX discretizations
for general multiphysics systems.
Specific implicit-explicit decompositions and consistent predictors are designed to allow
the monolithic discretization to be solved in a partitioned manner, i.e., subsystem-by-subsystem,
and meanwhile, maintain arbitrarily high-order accuracy and {reasonable} stability properties.
{However, due to the explicit component in the IMEX schemes,
these partitioned solvers cannot handle differential-algebraic systems, due to the singular mass matrix.}

This work extends the work in \cite{Huang2019}, and presents arbitrarily high-order {\it partitioned} spectral deferred correction
schemes for general multiphysics systems.
The SDC method, first proposed in \cite{Dutt2000}, is a general class of methods for solving initial value problems determined by
ordinary differential equations (ODEs), wherein high-order accuracy is attained by performing a series of correction
sweeps using a low-order time-stepping method. Implicit versions of this method are shown to
have good stability even for stiff equations~\cite{minion2003semi, bourlioux2003high, hagstrom2006on}. One of the most attractive features of SDC is the flexibility in the choice of the low order solver for the correction equation.
As for multiphysics systems, when a partitioned low-order solver is chosen,  the proposed multiphysics solver can be both arbitrarily high-order accurate and  partitioned.
In the present work, these low-order partitioned solvers are designed using the weakly coupled Gauss-Seidel predictor proposed in \cite{Huang2019}, which features good stability. The accuracy and stability properties of these partitioned multiphysics solvers are analyzed both analytically and numerically.
The comparisons with the partitioned IMEX method in \cite{Huang2019} are presented, which suggest the high-order partitioned SDC solvers are more robust than the partitioned IMEX solvers, while the IMEX are more efficient for the
numerical examples considered in this work.
Moreover, it is worth mentioning, the present SDC scheme is capable of handling differential-algebraic {systems} of equations~(DAE), which is demonstrated in \cref{sec: modified cavity}.

The remainder of the paper is organized as follows. In
Section~\ref{sec: govern eq}, the general form of the multiphysics problem
as a system of $m$ systems of partial differential equations and its
semi-discretization are introduced. In Section~\ref{sec: sdc},
an overview of SDC schemes is provided. In Section~\ref{sec: sdc partitioned solver},
the arbitrarily high-order SDC solvers are introduced and their features such as accuracy and stability are discussed.
Numerical applications are provided in
Section~\ref{sec: app}
that demonstrate the high-order accuracy and good stability properties of the
proposed solvers on an ODE system, an advection-diffusion-reaction system, and fluid-structure
interaction problems with both incompressible flows and compressible flows.

\section{Governing multiphysics equations and semi-discrete formulation}
\label{sec: govern eq}
As in the works \cite{Huang2019,Huang2019b}, we consider a general formulation
for multiple interacting physical processes, described by a coupled system of
partial differential equations,
\begin{equation} \label{eq:multiphys}
   \frac{\partial  u^i}{\partial t}
      = \L^i (  u^i,  c^i, \bm x, t ),
      \quad \bm x \in \Omega^i(c^i),
      \quad t \in [0, T],
\end{equation}
for $1 \leq i \leq \nphys$, where $\nphys$ denotes the number of physical
subsystems. Appropriate boundary conditions, omitted here for brevity, are enforced
for each of the subsystems. The $i$th physical subsystem is modeled as a partial
differential equation with corresponding differential operator denoted as $\L^i$.
The state variable $u^i(\bm x, t)$ denotes the solution to the $i$th
equation in the spatial domain $\Omega^i$, in the time interval $[0,T]$. The
coupling between the physical subsystems is described through the \textit{coupling
term } $c^i(u^1, \ldots, u^{\nphys}, \bm x, t)$, which couples the
$i$th subsystem to the $\nphys - 1$ remaining subsystems. In the most general case,
the differential operator $\L^i$, spatial domain $\Omega^i$, and boundary
conditions all depend on this coupling term.

In this work, we are concerned with the temporal integration of the coupled
system \cref{eq:multiphys}. We first begin by assuming a spatial discretization
for each of the physical subsystems, and then writing a general semi-discretized form
for the $i$th subsystem as a system of ordinary differential equations
\begin{equation} \label{eq:semi-discrete-multiphys}
   \mass^i \frac{\partial \bm u^i}{\partial t}
   = \bm r^i (\bm u^i, \bm c^i, t), \quad t \in [0,T],
\end{equation}
where $\bm M^i$ denotes the fixed mass matrix and $\bm r^i$ denotes the spatial residual
corresponding to a spatial discretization of the problem. Here we use the
notation $\bm u^i$ to denote the discretized solution represented
as a vector of degrees of freedom. In general, the coupling term $\bm c^i$ will
result in a coupling between all $\nphys$ subsystems of ODEs given by
\cref{eq:semi-discrete-multiphys}. We can write this large, coupled system of
equations in the simple form as
\begin{equation} \label{eq:semi-discrete}
   \mass \frac{\partial \bm u}{\partial t}
      = \bm r(\bm u, \bm c, t), \quad t \in [0,T],
\end{equation}
where $\bm u$, $\bm c$, and $\bm r$ represent the state vectors, coupling terms,
and spatial residuals for each of the single-physics subsystems concatenated as
\begin{equation}
   \bm u = \left( \begin{array}{c}
      \bm u^1 \\ \vdots \\ \bm u^\nphys
   \end{array} \right), \qquad
   \bm c = \left( \begin{array}{c}
      \bm c^1 \\ \vdots \\ \bm c^\nphys
   \end{array} \right), \qquad
   \bm r = \left( \begin{array}{c}
      \bm r^1(\bm u^1, \bm c^1, t) \\
      \vdots \\
      \bm r^\nphys(\bm u^\nphys, \bm c^\nphys, t)
   \end{array} \right).
\end{equation}
The mass matrix $\mass$ is considered to be a block-diagonal matrix with the
single-physics mass matrices $\mass^i$ along the diagonal,
\begin{equation}
   \mass = \left( \begin{array}{ccc}
      \mass^1 \\
      & \ddots \\
      & & \mass^\nphys
   \end{array} \right).
\end{equation}

In order to construct \textit{partitioned} time integration schemes for the
system \cref{eq:semi-discrete}, we write the total derivative of the
spatial residuals $\bm r$ as
\begin{equation}
   D_{\bm u} \bm r = \frac{\partial \bm r}{\partial \bm u}
      + \frac{\partial \bm r}{\partial \bm c}
        \frac{\partial \bm c}{\partial \bm u}.
\end{equation}
The terms on the right-hand side of this equation are Jacobian matrices with
block structures given by
\begin{equation}
\frac{\partial \bm r}{\partial \bm u} =
   \left( \begin{array}{ccc}
      \frac{\partial \bm r^1}{\partial \bm u^1} \\
      & \ddots \\
      & & \frac{\partial \bm r^\nphys}{\partial \bm u^\nphys}
   \end{array} \right), \quad
\frac{\partial \bm r}{\partial \bm c} =
   \left( \begin{array}{ccc}
      \frac{\partial \bm r^1}{\partial \bm c^1} \\
      & \ddots \\
      & & \frac{\partial \bm r^\nphys}{\partial \bm c^\nphys}
   \end{array} \right), \quad
\frac{\partial \bm c}{\partial \bm u} =
   \left( \begin{array}{ccc}
      \frac{\partial \bm c^1}{\partial \bm u^1}
         & \cdots
         & \frac{\partial \bm c^1}{\partial \bm u^\nphys}  \\
      \vdots & \ddots & \vdots \\
      \frac{\partial \bm c^\nphys}{\partial \bm u^1}
         & \cdots
         & \frac{\partial \bm c^\nphys}{\partial \bm u^\nphys}
   \end{array} \right).
\end{equation}
The term $\frac{\partial \bm r}{\partial \bm u}$ is block diagonal, and
represents the contribution of a given physics state to its own subsystem. The
second term, $\frac{\partial \bm r}{\partial \bm c}\frac{\partial \bm
c}{\partial \bm u}$ represents the coupling between subsystems.

\begin{remark}
The coupled system of ODEs in
(\ref{eq:semi-discrete-multiphys}) { or}  (\ref{eq:semi-discrete})
is the starting point for the mathematical formulation of the
proposed high-order, partitioned SDC method; therefore the
method applied to problems directly modeled as a system of
ODEs in addition to ODEs that result from semi-discretization
of a system of PDEs.
\end{remark}

\section{Spectral Deferred Corrections} \label{sec: sdc}

Spectral deferred correction methods are a class of numerical methods for
approximating the solution to ordinary differential equations through an
iterative process based on the Picard equation \cite{Dutt2000}. These methods
have garnered a large amount of interest~\cite{minion2003semi, hagstrom2006on, causley2019convergence}, and have been applied to a wide
variety of {problems~\cite{bourlioux2003high, christlieb2014high, crockatt2017arbitrary, minion2018higher, Pazner2016}}. An attractive feature of SDC methods is that they
are capable of arbitrary formal order of accuracy. Additionally, their
implementation is relatively straightforward, since they are typically built by
combining simple low-order methods, such as forward Euler or backward Euler.
Additionally, and most importantly for this work, the iterative nature of SDC
methods is very flexible, allowing for sophisticated semi- and multi-implicit
splitting schemes \cite{minion2003semi,bourlioux2003high,Pazner2016}. It is
this flexibility that will allow us to construct efficient partitioned
multiphysics integrators.

We begin by considering an ordinary differential equation given by
\begin{equation}
\label{eq:sdc_ode}
   \frac{d \bm{u}(t)}{dt} = \bm{r}(\bm{u}, t),
\end{equation}
with initial condition $\bm{u}(t_n) = \bm{u}_n$. The SDC method is a one-step method
to advance the solution from $t_n$ to $t_{n+1} = t_n + \Delta t$. Integrating
from $t_n$ to $t$ (for arbitrary $t > t^n$), we obtain the associated integral
equation,
\begin{equation}
   \bm{u}(t) = \bm{u}_n + \int_{t_n}^t \bm{r}(\bm{u}(\tau), \tau) \, d\tau.
\end{equation}
For the sake of brevity, we will omit the dependence of $\bm{r}$ on $t$. As the
SDC method is an iterative process, let $k$ denote the iterative index and $\bm{u}^{(k)}(t)$ denote an approximation to the solution $\bm{u}(t)$ of the $k$th
iterate.
The SDC method seeks to obtain an improved approximation $\bm{u}^{(k+1)}(t)$ by
approximating the solution to the correction equation
\begin{equation} \label{eq:sdc-correction}
   \bm{u}^{(k+1)}(t) = \bm{u}_n
      + \int_{t_n}^t \left( \bm{r}(\bm{u}^{(k+1)}) - \bm{r}(\bm{u}^{(k)}) \right) \, d\tau
      + \int_{t_n}^t \bm{r}(\bm{u}^{(k)}) \, d\tau.
\end{equation}
In order to obtain the SDC method, this correction equation is discretized by
replacing the integrals with approximations computed using quadrature rules. The
first integral on the right-hand side of \cref{eq:sdc-correction} is
discretized using a low-order method (with order of accuracy $p_{\rm low}$,
typically $p_{\rm low} = 1$). This low-order method usually corresponds to
forward or backward Euler. The second integral is approximated using a
high-order quadrature rule with order of accuracy $p_{\high}$. Each iteration
updates the solution from $\bm{u}^{(k)}$ to $\bm{u}^{(k+1)}$, improving the order of
accuracy of the provision solution by $p_{\rm low}$, up to a maximum of $p_{\rm
high}$ \cite{Dutt2000}.

We begin by selecting a high-order accurate quadrature rule on the interval
$[t^n, t^n + \Delta t]$. The order of accuracy of this quadrature rule,
denoted $p_{\high}$, is equal to the formal order of accuracy of the
resulting SDC method. The abscissas of this quadrature rule, which we denote
\begin{equation} \label{eq:temporal-nodes}
   t_n = t_{n,\,0} < t_{n,\,1} < \cdots < t_{n,\,q} = t_n + \Delta t,
\end{equation}
can be considered to be nodes at which we approximate the solution to the ODE.
For simplicity of notation, we have included the left and right
endpoints of the interval as points in the quadrature rule. This choice is made
for uniformity of notation, and more general quadrature rules can be considered
by assigning zero quadrature weights to one or both of the endpoints. The
temporal nodes \cref{eq:temporal-nodes} give rise to $q$ time sub-steps
$[t_{n,\,j}, t_{n,\,j+1}]$ for $0 \leq j \leq q-1$, with $\Delta t_{n,\,j} = t_{n,\,j+1} - t_{n,\,j}$.
Given a function whose value is known at each of the temporal nodes, its
integral over the sub-interval can be approximated by integrating the resulting
interpolating polynomial. Thus, given a function $\psi(t)$ and nodal values
$\psi_i = \psi(t_{n,\,i})$, we introduce the notation $I_j^{j+1} \psi$ to denote
the resulting approximation to $\int_{t_{n,\,j}}^{t_{n,\,j+1}} \psi(\tau) \, d\tau$,
\begin{equation}
\label{eq:integration}
I_j^{j+1} \psi = \int_{t_{n,\,j}}^{t_{n,\,j+1}} \psi(\tau) \, d\tau = \sum_{i = 0}^{q} w^j_i \psi_i,
\end{equation}
here $w^{j}_i$ is the weight related to abscissa $t_{n,\,i}$.

Given a previous approximation $\bm{u}^{(k)}_{n,\,j}$ for $0 \leq j \leq q$ at the $k$th iteration
and a current approximation $\bm{u}^{(k+1)}_{n,\,j}$, the SDC method produces an updated
approximation at the next temporal node $\bm{u}^{(k+1)}_{n,\,j+1}$ by discretizing
\cref{eq:sdc-correction}. In particular, if forward Euler is used for the
low-order quadrature rule, the solution at the next temporal node is given by
the following \textit{explicit} update:
\begin{equation} \label{eq:sdc-fe}
\text{Forward Euler:} \qquad
   \bm{u}^{(k+1)}_{n,\,j+1} = \bm{u}^{(k+1)}_{n,\,j}
      + \Delta t_{n,\,j} \left( \bm{r}(\bm{u}^{(k+1)}_{n,\,j}) - \bm{r}(\bm{u}^{(k)}_{n,\,j}) \right)
      + I_j^{j+1} \bm{r}(\bm{u}^{(k)}_n).
\end{equation}
Similarly, if backward Euler is used for the low-order quadrature, the solution
at the next temporal node is given by the following \textit{implicit} update:
\begin{equation} \label{eq:sdc-be}
\text{Backward Euler:} \qquad
   \bm{u}^{(k+1)}_{n,\,j+1} = \bm{u}^{(k+1)}_{n,\,j}
      + \Delta t_{n,\,j} \left( \bm{r}(\bm{u}^{(k+1)}_{n,\,j+1}) - \bm{r}(\bm{u}^{(k)}_{n,\,j+1}) \right)
      + I_j^{j+1} \bm{r}(\bm{u}^{(k)}_n).
\end{equation}
Thus, given approximations $\bm{u}^{(k)}_{n,\,j}$ for $0 \leq j \leq q$, improved
approximations $\bm{u}^{(k+1)}_{n,\,j}$ are obtained through a sequence of forward or
backward Euler steps. Each update from $\bm{u}^{(k)}_n$ to $\bm{u}^{(k+1)}_n$,
termed an \textit{SDC sweep}, increases the order of accuracy of the solution by
$p_{\rm low}$, up to a maximum order of $p_{\high}$. Therefore, when using
backward or forward Euler corrections, $p_{\high}$ iterations are generally
required to achieve a formal order accuracy of $p_{\high}$. Starting this
process requires an initial guess for the solution $\bm{u}^{(0)}_{n,\,j}$ for $0 \leq j \leq q$.
Typically it is sufficient to use the previous step solution $\bm{u}^{(0)}_{n,\,j} = \bm{u}_n$ as
an initial guess.

It can be useful to note that the SDC iterations can be viewed as a fixed-point
iteration, converging to the collocation solution $\bm{u}^{\rm col}_j$, which
satisfies
\begin{equation} 
   \bm{u}^{\rm col}_{n,\,j+1} = \bm{u}^{\rm col}_{n,\,j} + I_j^{j+1} \bm{r}(\bm{u}_{n}^{\rm col}).
\end{equation}
Collocation schemes have been studied extensively in the context of
fully-implicit Runge-Kutta methods \cite{Hairer1996,Pazner2017,causley2019convergence}.
These methods typically have very attractive stability and accuracy properties;
however, solving the resulting algebraic systems may be challenging.

\section{Partitioned SDC schemes for multiphysics}
\label{sec: sdc partitioned solver}
In this work, we consider implicit SDC methods corresponding
to the discretized system \cref{eq:sdc-be}. A straightforward application of
this scheme to the multiphysics system \cref{eq:semi-discrete} results in
\begin{equation} \label{eq:sdc-partitioned}
   \bm{M}^i \bm u^{i, (k+1)}_{n,\,j+1} = \bm{M}^i \bm u^{i,(k+1)}_{n,\,j}
      + \Delta t_{n,\,j} \Big(
         \bm r^i( \bm u^{i,(k+1)}_{n,\,j+1}, \bm c^{i,(k+1)}_{n,\,j+1})
       - \bm r^i( \bm u^{i,(k)}_{n,\,j+1}, \bm c^{i,(k)}_{n,\,j+1}) \Big)
      + I_j^{j+1} \bm r( \bm u^{i,(k)}_n, \bm c^{i,(k)}_n).
\end{equation}
Here the first superscript $i$ represents the subsystem number, the
second superscript $k$ represents the iteration number,  and the
subscript $n$ and $j$ represent the time step and the abscissa. The main challenge in solving
the resulting system of equations is the coupling term $\bm c^{i,(k+1)}_{n,\,j+1}$,
which, in general, results in a fully-coupled system of equations. In
order to reduce this coupling, we introduce certain approximations
to this coupling term based on the predictors introduced in \cite{Huang2019}.
The main
idea of the present approach is to reduce the coupling by making use of the
state variables from the previous \textit{iterate} when evaluating the coupling
term. This is in contrast to the IMEX methods presented in \cite{Huang2019},
which lagged the coupling terms one time step to maintain the design order
of the IMEX scheme. We consider only the weak Gauss-Seidel type predictor
$\bm{\tilde c}^{i,(k+1)}_{n,\,j+1}$ for the $i$th subsystem, defined as follows:
\begin{align} \label{eq:weak-gs}
   \bm{\tilde c}^{i,(k+1)}_{n,\,j+1} &= \bm c(\bm u^{1,(k+1)}_{n,\,j+1},
      \ldots, \bm u^{i-1,(k+1)}_{n,\,j+1}, \bm u^{i,(k)}_{n,\,j+1}, \ldots, \bm u^{\nphys,(k)}_{n,\,j+1}) ,
\end{align}
which depends on the the most up-to-date information from the previous $i-1$
subsystems. This choice of predictor implies the \textit{ordering} of the
subsystems is important; see \cite{Huang2019} for a general discussion of the
ordering of subsystems in the context of Gauss-Seidel predictors for
multiphysics partitioned solvers and Section~\ref{sec: app}
for the ordering used for the applications considered in this work.

This choice of predictor is quite simple to implement, and its robustness has
been demonstrated in \cite{Huang2019}. The partitioned solver is then
constructed by replacing the term $\bm c_{n,\,j+1}^{i,(k+1)}$ in
\cref{eq:sdc-partitioned} with its appropriately chosen approximation
$\bm{\tilde c}^{i,(k+1)}_{n,\,j+1}$. The detailed algorithm is summarized in
Algorithm~\ref{ALG: SDC WEAK GS}. It is also worth mentioning that to solve the
system \cref{eq:sdc-partitioned} for $\bm u^{i, (k+1)}_{n,\,j+1}$ with the approximation
\cref{eq:weak-gs}, only the single-physics Jacobian
$\displaystyle{\frac{\partial \bm r^{i}}{\partial \bm u^i}}$ is required
because $\bm{\tilde c}^{i,(k+1)}_{n,\,j+1}$ does not depend on $\bm u^{i,(k+1)}_{n,\,j+1}$.

\begin{algorithm}
 \caption{Spectral deferred correction partitioned multiphysics scheme}
 \label{ALG: SDC WEAK GS}
 \begin{algorithmic}[1]
  \State Set $0$th sweep values $\bm u_{n,\,j}^{i, (0)} = \bm u^{i}_{n}$  for $i = 1, ..., m$,  $j = 0,\,\dots,\,q$
   \For{iterations $k = 0,\,\dots,\,p_{\high}-1$}
    \State Set  $\bm u_{n,\,0}^{i, (k)} = \bm u^{i}_{n}$ for $i = 1, ..., m$
   \For{abscissas $j = 0,\,\dots,\,q-1$}
    \For{physical subsystems $i = 1,\,\dots,\,m$}
     \State Implicit solve for $\bm u^{i, (k+1)}_{n,\,j+1}$:
\[
       ~~~~~~~~~~\mass^i \bm u^{i, (k+1)}_{n,\,j+1} = \mass^i \bm u^{i,(k+1)}_{n,\,j}
      + \Delta t_{n,\,j} \Big(
         \bm r^i( \bm u^{i,(k+1)}_{n,\,j+1}, \bm{\tilde c}^{i,(k+1)}_{n,\,j+1})
       - \bm r^i( \bm u^{i,(k)}_{n,\,j+1}, \bm c^{i,(k)}_{n,\,j+1}) \Big)
      + I_j^{j+1} \bm r( \bm u^{i,(k)}_n, \bm c^{i,(k)}_n) \label{ALG:SOLVE}
\]
    \EndFor
   \EndFor
   \EndFor
   \State Set $\bm u^{i}_{n+1} = u^{i, p_{\high}}_{n,\,q}$ for $i = 1, ..., m$
 \end{algorithmic}
\end{algorithm}

{
\subsection{Partitioned SDC schemes}
\label{sec:sdc_schemes}
Based on the discussion above, we build a family of partitioned SDC schemes by choosing different quadrature points for \cref{ALG: SDC WEAK GS},
which are listed as follows
\begin{enumerate}
\item The first-order scheme~($p_{\high} = 1$), with $2$ abscissas $\{0, 1\}$. The corresponding integrals in \cref{eq:integration} are defined as
\begin{equation*}
I_0^1 \psi = \Delta t \psi(1),
\end{equation*}
which is abbreviated as SDC1.
\item The second-order scheme~($p_{\high} = 2$), with $2$ abscissas $\{0, 1\}$. The corresponding integrals in \cref{eq:integration} are defined as
\begin{equation*}
I_0^1 \psi = \frac{ \Delta t}{2}\psi(0) +  \frac{ \Delta t}{2}\psi(1),
\end{equation*}
which is abbreviated as SDC2.
\item The third-order scheme~($p_{\high} = 3$), with $3$ Gauss-Radau abscissas $\left\{0, \tfrac{1}{3}, 1\right\}$. The corresponding integrals in \cref{eq:integration} are defined as
\begin{equation*}
I_0^1 \psi = \frac{5 \Delta t}{12}\psi(\frac{1}{3})  - \frac{\Delta t}{12}\psi(1) \quad \textrm{and}
\quad I_1^2 \psi = \frac{\Delta t}{3}\psi(\frac{1}{3})  + \frac{\Delta t}{3}\psi(1),
\end{equation*}
which is abbreviated as SDC3-r, for this case, the lower order approximation is specifically chosen~(see \cref{remark_accuracy})  as
\begin{equation*}
\mathcal{C}_j = \Delta t \Big(
         \bm r^i( \bm u^{i,(k+1)}_{n,\,j+1}, \bm{\tilde c}^{i,(k+1)}_{n,\,j+1})
       - \bm r^i( \bm u^{i,(k)}_{n,\,j+1}, \bm c^{i,(k)}_{n,\,j+1})\Big),
       \end{equation*}
where $\Delta t$ is used instead of $\Delta t_{n,\,j}$.
\item The third-order scheme~($p_{\high} = 4$), with $3$ Gauss-Lobatto abscissas $\left\{0, \tfrac{1}{2}, 1\right\}$. The corresponding integrals in \cref{eq:integration} are defined as
\begin{equation*}
I_0^1 \psi = \frac{5 \Delta t}{24}\psi(0)  + \frac{8 \Delta t}{24}\psi(\frac{1}{2}) - \frac{ \Delta t}{24}\psi(1) \quad \textrm{and}
\quad I_1^2  = -\frac{ \Delta t}{24}\psi(0)  + \frac{8 \Delta t}{24}\psi(\frac{1}{2}) + \frac{5 \Delta t}{24}\psi(1),
\end{equation*}
which is abbreviated as SDC3-l. This scheme uses a fourth-order quadrature, but only three SDC sweeps.
\item The fourth-order scheme~($p_{\high} = 4$), with $3$ Gauss-Lobatto abscissas $\left\{0, \frac{1}{2}, 1\right\}$. The corresponding integrals in \cref{eq:integration} are defined as
\begin{equation*}
I_0^1 \psi = \frac{5 \Delta t}{24}\psi(0)  + \frac{8 \Delta t}{24}\psi(\frac{1}{2}) - \frac{ \Delta t}{24}\psi(1) \quad \textrm{and}
\quad I_1^2  = -\frac{ \Delta t}{24}\psi(0)  + \frac{8 \Delta t}{24}\psi(\frac{1}{2}) + \frac{5 \Delta t}{24}\psi(1),
\end{equation*}
which is abbreviated as SDC4.
\end{enumerate}
}

\subsection{Accuracy of the partitioned SDC schemes} \label{sec:sdc-accuracy}
To analyze the order of accuracy of our partitioned SDC schemes, let $\bm{u}(t)$ be the exact solution of \cref{eq:sdc_ode}, which satisfies
\begin{equation}
\label{eq:acc_exact}
\bm{u}(t_{n,\,j+1}) = \bm{u}(t_{n,\,j}) + \int_{t_{n,\,j}}^{t_{n,\,j+1}} \bm{r}(\bm{u}(\tau), \tau) \, d\tau,
\end{equation}
here we assume the function $\bm{r}$ is $C^{1}$ continuous, which is sufficient to guarantee the local existence and uniqueness of the solution. Let $L$ denote the Lipschitz constant of $\bm{r}$.

The update equations~{(\cref{eq:sdc-fe}, \cref{eq:sdc-be}, or \cref{ALG:SOLVE} in \cref{ALG: SDC WEAK GS})} are written in a general form as
\begin{equation}
\label{eq:acc_general}
	\bm{u}^{(k+1)}_{n,\,j+1} = \bm{u}^{(k+1)}_{n,\,j} + \mathcal{C}_j(\bm{u}^{(k+1)}_n, \bm{u}^{(k)}_n)
      + I_j^{j+1} \bm{r}(\bm{u}^{(k)}_n),
\end{equation}
where $\mathcal{C}_j(\bm{u}^{(k+1)}_n, \bm{u}^{(k)}_n)$ denotes the low-order approximation from $t_{n,j}$ to $t_{n,j+1}$.
We now give the {\it local error} of SDC schemes by induction, on the assumption that the numerical solution at the previous solution point $t_n$ is exact.
This result extends standard convergence results in the SDC literature \cite{Hansen2011,causley2019convergence,Tang2012} to the case of the partitioned SDC presented here.

\begin{thm} \label{thm:local-error} 
We assume that the low-order correction $\mathcal{C}_j(\cdot,\,\cdot)$ satisfies the following Lipschitz-type conditions:
\begin{equation}
\begin{aligned}
\label{eq:sufficient_condition}
\left \| \mathcal{C}_j(\bm{u}^{(k+1)}_n, \bm{u}^{(k)}_n) - \mathcal{C}_j(\bm{u}(t_{n,j}), \bm{u}^{(k)}_n) \right\| &\leq C \Delta t L \sum_{j'\leq j+1} \| \bm{u}^{(k+1)}_{n, j'} - \bm{u}(t_{n,j'}) \|,\\
\left \| \mathcal{C}_j(\bm{u}(t_{n,j}), \bm{u}^{(k)}_n) \right\| &\leq C \Delta t L \sum_{j'\leq j+1} \| \bm{u}(t_{n, j'}) - \bm{u}^{(k)}_{n,j'} \|,
\end{aligned}
\end{equation}
where the constant $C$ is independent of $j$, $k$, and $\Delta t$.
And the function $\bm{r}$ is $C^{p_{\high}}$ continuous.
Then, we have the following estimate for the local error of the SDC scheme:
\begin{equation}
\label{eq:acc_result}
         \| \bm{u}_{n,\,j}^{(k)} - \bm{u}(t_{n,\,j}) \| = \mathcal{O}\left(\Delta t^{\min\{k+1,p_{\high}+1\}}\right) \quad \text{for} \quad 0 \leq j \leq q,\, 0\leq k.
\end{equation}
\end{thm}
\begin{proof} 
We proceed by induction on $j$ and $k$.
It is easy to verify that \cref{eq:acc_result} holds for base cases with $j=0$ or $k=0$.
Subtracting  \cref{eq:acc_exact} from \cref{eq:acc_general}, we have
\begin{equation}
\label{eq:acc_subtract}
	\bm{u}^{(k+1)}_{n,\,j+1} - \bm{u}(t_{n,\,j+1})
	= \bm{u}^{(k+1)}_{n,\,j} - \bm{u}(t_{n,\,j}) + \mathcal{C}_j(\bm{u}^{(k+1)}_n, \bm{u}^{(k)}_n)
      + I_j^{j+1} \bm{r}(\bm{u}^{(k)}_n) - \int_{t_{n,\,j}}^{t_{n,\,j+1}} \bm{r}(\bm{u}(\tau)) \, d\tau.
\end{equation}
Since the underlying quadrature rule has order of accuracy $p_{\high}$, we have
\begin{equation} \label{eq:integ_approx}
I_j^{j+1} \bm{r}(\bm{u}^{(k)}_n) - \int_{t_{n,\,j}}^{t_{n,\,j+1}} \bm{r}(\bm{u}(\tau)) \, d\tau
 = I_j^{j+1} \left[ \bm{r}(\bm{u}^{(k)}_n) - \bm{r}(\bm{u}(t)) \right]
 + \Delta t_{n,j}\mathcal{O}(\Delta t^{p_{\high}}).
\end{equation}
By induction, we assume \cref{eq:acc_result} holds for all $k' \leq k$ and $k^{'} = k+1,\, j' \leq j$.
Using the inductive assumption, we have
\begin{equation} \label{eq:induction}
\bm{u}^{(k+1)}_{n,\,j} - \bm{u}(t_{n,\,j}) = \mathcal{O}\left(\Delta t^{\min\{k+2, p_{\high}+1\}}
\right).
\end{equation}
Making use of \cref{eq:integ_approx} and \cref{eq:induction}, we see that \cref{eq:acc_subtract} is then reduced to
\begin{equation}
\bm{u}^{(k+1)}_{n,\,j+1} - \bm{u}(t_{n,\,j+1})
   = \mathcal{C}_j(\bm{u}^{(k+1)}_n, \bm{u}^{(k)}_n) + \mathcal{O}(\Delta t^{\min\{k+2, p_{\high}+1\}}).
\end{equation}
By using property \eqref{eq:sufficient_condition} and induction on $k$, we have
\begin{align*}
\| \mathcal{C}_j(\bm{u}^{(k+1)}_n, \bm{u}^{(k)}_n) \|
   &= \| \mathcal{C}_j(\bm{u}^{(k+1)}_n, \bm{u}^{(k)}_n) - \mathcal{C}_j(\bm{u}, \bm{u}^{(k)}_n) + \mathcal{C}_j(\bm{u}, \bm{u}^{(k)}_n) \| \\
   &\leq \sum_{j'\leq j+1} C \Delta t L \| \bm{u}^{(k+1)}_{n,j'} - \bm{u}(t_{n,j'}) \| + \Delta t \mathcal{O}(\Delta t^{\min\{k+1, p_{\high}+1\}}),
\end{align*}
Now, by induction on $j$, we obtain
\[
   (1 - C \Delta t L) \|\bm{u}^{(k+1)}_{n,\,j+1} - \bm{u}(t_{n,\,j+1}) \|
    = \sum_{j' \leq j} C \Delta t L \| \bm{u}^{(k+1)}_{n,j'} - \bm{u}(t_{n,j'}) \| + \mathcal{O}(\Delta t^{\min\{k+2, p_{\high}+1\}})
    = \mathcal{O}(\Delta t^{\min\{k+2, p_{\high}+1\}}).
\]
Therefore, \cref{eq:acc_result} holds for all $k \geq 0$ and $0 \leq j \leq q$.
This finishes our proof of \cref{eq:acc_result}, which also indicates the optimal global error of SDC schemes is  $\mathcal{O}(\Delta t^{p_{high}})$, when $p_{high}$ sweeps are applied.
\end{proof}

\begin{remark}
As for the hypothesis \eqref{eq:sufficient_condition} of Theorem \ref{thm:local-error}, it is easy to verify that the forward Euler approximation~\cref{eq:sdc-fe}, backward Euler approximation~\cref{eq:sdc-be} and our weak Gauss-Seidel predictor based approximation in~\cref{ALG: SDC WEAK GS} all satisfy the sufficient conditions.
Therefore, these three approximations can all lead to design order of accuracy.
\end{remark}

\begin{remark}  For these three schemes, even if we change $\Delta t_{n,\,j}$ to $\alpha_j \Delta t_{n,\,j}$~{(See SDC3-r)}, for any $\alpha_j = \mathcal{O}(1)$ in the low order approximation $\mathcal{C}_j$, the sufficient condition~\eqref{eq:sufficient_condition} still holds.
That means to achieve arbitrary high order accuracy, the low order approximation $\mathcal{C}_j$ in SDC schemes is not required to be an accurate approximation, which gives more flexibility to design new stable schemes.
\label{remark_accuracy}
\end{remark}
\subsection{Stability of the partitioned SDC schemes}
\label{sec:sdc-stability}
{The stability properties of SDC schemes have been {\it{numerically}} analyzed widely in \cite{Dutt2000, bourlioux2003high, causley2019convergence}, in which L-stable and A-stable properties are reported.
 We will analyze the stability of the partitioned SDC schemes based on a model linear system

\begin{equation}
\dot\stvc = \bm{A} \stvc
\end{equation}
where $\bm{A} = \bm{L} + \bm{D} + \bm{U}$ is an $n \times n$ matrix, $\bm{L}$ is the lower triangular part of $\bm{A}$, $\bm{U}$ is the upper triangular
part of $\bm{A}$, and $\bm{D}$ is the diagonal of $\bm{A}$. The system is treated as $n$ subsystems and the coupling term is taken as $\cbm(\ubm) = (\bm{L}+\bm{U})\ubm$
We will prove SDC1 scheme is unconditionally stable, when $\bm{A}$ is strictly diagonally
dominant with non-positive diagonal entries.
The update matrix for the SDC1 takes the form
\begin{equation}
\label{eq:update_matrix_SDC1}
 \bm{C} = (\bm{I} - \dt{}\bm{L} - \dt{}\bm{D})^{-1}
                  (\bm{I} + \dt{}\bm{U}).
\end{equation}
Any of its eigenpairs $(\lambda, \xbm)$ with $|\lambda| \geq 1$ satisfy the relation
\begin{equation}
 (\bm{I} - \dt{}\bm{L} - \dt{}\bm{D})^{-1}
                   (\bm{I} + \dt{}\bm{U}) \xbm =
 \lambda \xbm,
\end{equation}
which can be re-arranged as
\begin{equation}
 (\bm{I} + \dt{}\bm{U}) \xbm = (\bm{I} - \dt{}\bm{L} - \dt{}\bm{D}) \lambda \xbm
\end{equation}
or written as components as
\begin{equation}
 x_i + \dt{}\sum_{j>i} a_{ij}x_j + \dt{}\lambda \sum_{j<i} a_{ij} x_j =
 \lambda x_i - \dt{}\lambda a_{ii}x_i
\end{equation}
for $i = 1,\dots,N$. Application of the triangular inequality and division by
$|x_i|$ leads to the relation
\begin{equation} \label{eqn:app:eq0}
 1 + \dt{}\sum_{j>i} |a_{ij}| \frac{|x_j|}{|x_i|} +
 \dt{}|\lambda| \sum_{j<i} |a_{ij}| \frac{|x_j|}{|x_i|} \geq
 |\lambda| |1 - \dt{} a_{ii}|.
\end{equation}
The assumption of strictly diagonal dominance and negative diagonal entries leads to
the following bound
\begin{equation} \label{eqn:app:ineq1}
 |\lambda| |1 - \dt{} a_{ii}| =
 |\lambda| (1 + \dt{} |a_{ii}|) >
 |\lambda| (1 + \dt{} \sum_{j\neq i} |a_{ij}|).
\end{equation}
On the other hand, if $i = \arg\max_{1\leq j\leq n} |x_j|$, (\ref{eqn:app:eq0})
leads to
\begin{equation} \label{eqn:app:ineq2}
 1 + \dt{}\sum_{j>i} |a_{ij}| + \dt{}|\lambda| \sum_{j<i} |a_{ij}|   \geq  |\lambda| |1 - \dt{} a_{ii}| .
\end{equation}
Combining (\ref{eqn:app:ineq1}) and (\ref{eqn:app:ineq2}), we arrive at
\begin{equation}
 1 + \dt{} \sum_{j>i}|a_{ij}| > |\lambda|(1 + \dt{} \sum_{j>i} |a_{ij}|), 
\end{equation}
which leads to the desired
result
\begin{equation}
 \rho(\bm{C}) < 1
\end{equation}
and confirms that, under the stated assumptions, the SDC1 scheme is unconditionally stable.

\begin{remark} The condition can be relaxed as $\bm{A}$ is diagonally
dominant and irreducible with non-positive diagonal entries,
when you notice for any $|\lambda | \geq 1$,
if $\bm{A}$ is diagonally dominant and irreducible, then $\lambda(\bm{I} - \dt{}\bm{L} - \dt{}\bm{D}) - (\bm{I} + \dt{}\bm{U})$
is diagonally dominant and irreducible~\cite{bagnara1995unified}.
\end{remark}
}

\section{Applications}
\label{sec: app}

In this section, we present numerical results from a variety of multiphysics
systems for the proposed high-order, partitioned spectral deferred
correction solver. To demonstrate the high-order accuracy of the solver, we
consider a system of ODEs and the time-dependent
advection-diffusion-reaction equations. To test the robustness and applicability
of the method, we consider two fluid-structure interaction problems, including
both incompressible flows and compressible flows.

\subsection{Ordinary differential equations system}
\label{SEC: APP ODE}
{
In this section, we study the proposed high-order partitioned solvers on a stiff $2 \times 2$ system of linear ODEs
\begin{equation} \label{EQ: ODE SYS}
\dot\stvc = \bm{A} \stvc \,, \qquad
\bm{A} =
\begin{bmatrix}
 0 & 1 \\
 -\alpha & -\alpha - 1
\end{bmatrix}\,, \qquad
\ubm = \begin{bmatrix} \stvc[1] \\ \stvc[2] \end{bmatrix},
\end{equation}
with initial condition $\stvc(0) = (x_0,\,0)^{T}$ and consider the time
domain $t \in (0,\,20]$.
The eigenvalues of $\bm{A}$ are $-1$ and $-\alpha$, therefore, when $\alpha \gg 0$, the
system is very stiff. The exact solution at
any time $t$ is
\begin{equation}
\begin{bmatrix} \stvc[1](t) \\ \stvc[2](t) \end{bmatrix} =
\begin{bmatrix} x_0\Big(-\frac{1}{\alpha-1}\exp^{-\alpha t} + \frac{\alpha}{\alpha-1}e^{-t}\Big) \\
x_0\Big(\frac{\alpha}{\alpha -1}\exp^{-\alpha t} - \frac{\alpha}{\alpha-1}e^{-t}\Big)
\end{bmatrix}.
\end{equation}
To conform to the multiphysics formulation in \cref{eq:semi-discrete},
the ODE system is treated as a coupled system with three subsystems. The
mass matrix is identity, the residual term is taken as
\begin{equation}
\res = (\cpl[1], (-\alpha - 1)\stvc[2]+\cpl[2])^T,
\end{equation}
and the coupling terms are defined as
\begin{equation}
 \cpl[1] = \stvc[2]\,,\quad
 \cpl[2] = -\alpha \stvc[1].
\end{equation}
This decomposition of the residual term is non-unique. In fact,
many other choices exist that will lead to different schemes.

The maximum stable time steps $\Delta t_{\max}$ of monolithic forward Euler approach and the partitioned SDC1 approach~\footnote{A detailed analytical comparison of the monolithic forward Euler scheme, the partitioned SDC1 scheme, and the fixed point iteration scheme
on this model ODE problem is presented in \Cref{sec: proof}.} with respect to $\alpha$ are $\min\{\frac{2}{\alpha}, 2\}$ and $\frac{\alpha + 1 + \sqrt{(\alpha + 1)^2 + 4\alpha}}{\alpha}$. It is worth mentioning the maximum stable time step of SDC1 is independent of $\alpha$ and always greater than 1, which indicates our scheme is stable even some fast physical time scales are unresolved.

To test the stability and verify the temporal convergence of the partitioned schemes, we take $\alpha = 1000$ and $x_0 = 1000$.
Solutions till $t = 20.0$ obtained by these partitioned SDC schemes introduced in \cref{sec:sdc_schemes} with $\Delta t = 1.0$ are depicted in \cref{FIG: ODE SOL},
which indicates all schemes are stable.
The order of accuracy is quantified via the $L_\infty$-norm of the error in the numerical
solution at time $t = 20.0$
\begin{equation}
 e_\text{ODE} = \max_{1\leq i \leq2} |\pstp[i]{N} - \stvc[i](20)|,
\end{equation}
where $\stvc[i](20)$ is the exact solution at $t = 20.0$ and $\pstp[i]{N}$ the
numerical solution at the final time step for the $i$th subsystem. The
error $e_\text{ODE}$ as a function of the time step size for different SDC schemes are
shown in \cref{FIG: ODE 3FIELD}. The design order of
accuracy is achieved. It is worth mentioning that SDC3-r with modified low-order approximation can still lead to third order accuracy
as discussed in \cref{sec:sdc-accuracy}, but the error is slightly larger than that of SDC3-l scheme.

\begin{figure}[ht]
\centering
 \includegraphics[width=0.45\textwidth]{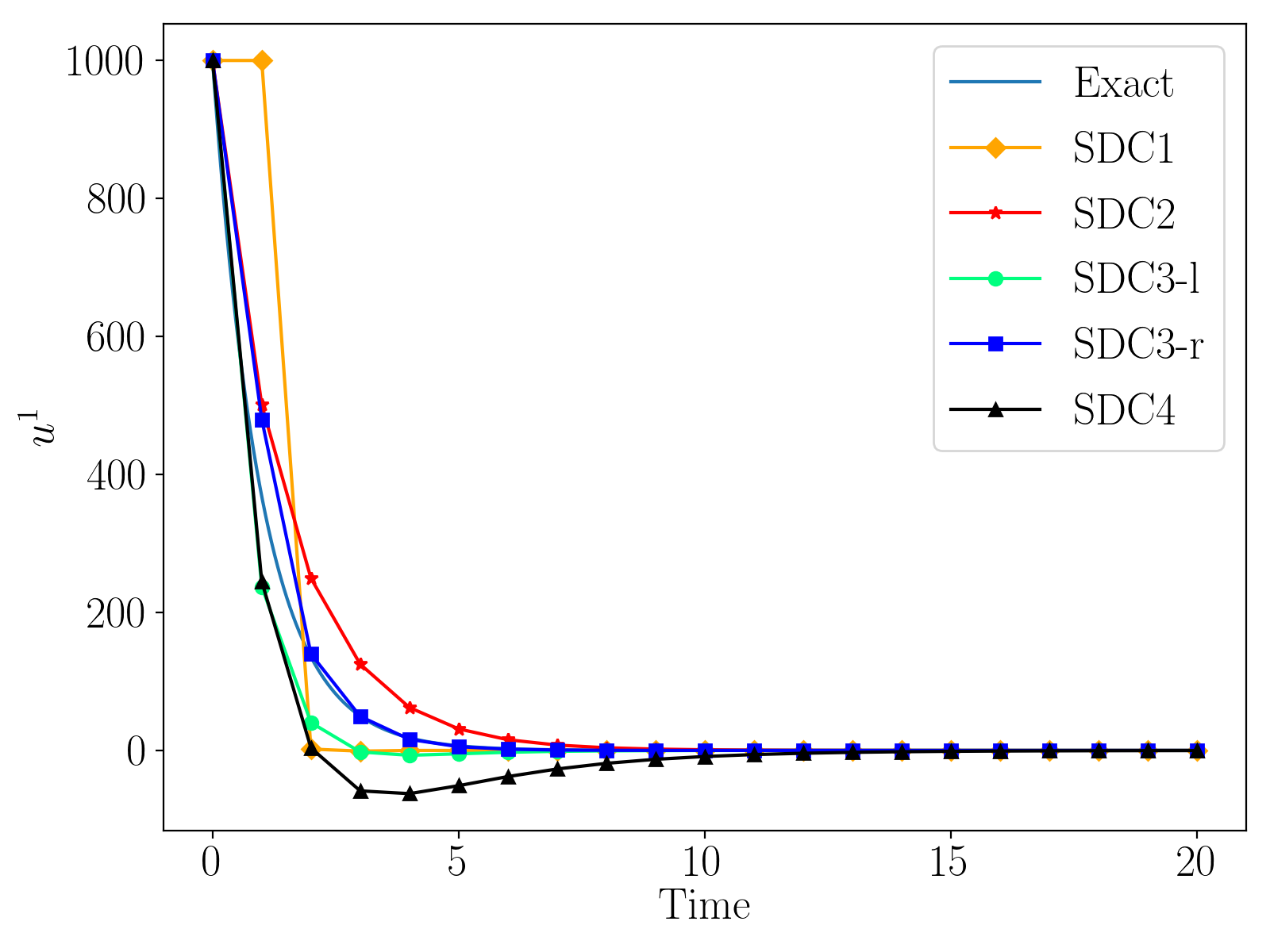} \quad
 \includegraphics[width=0.45\textwidth]{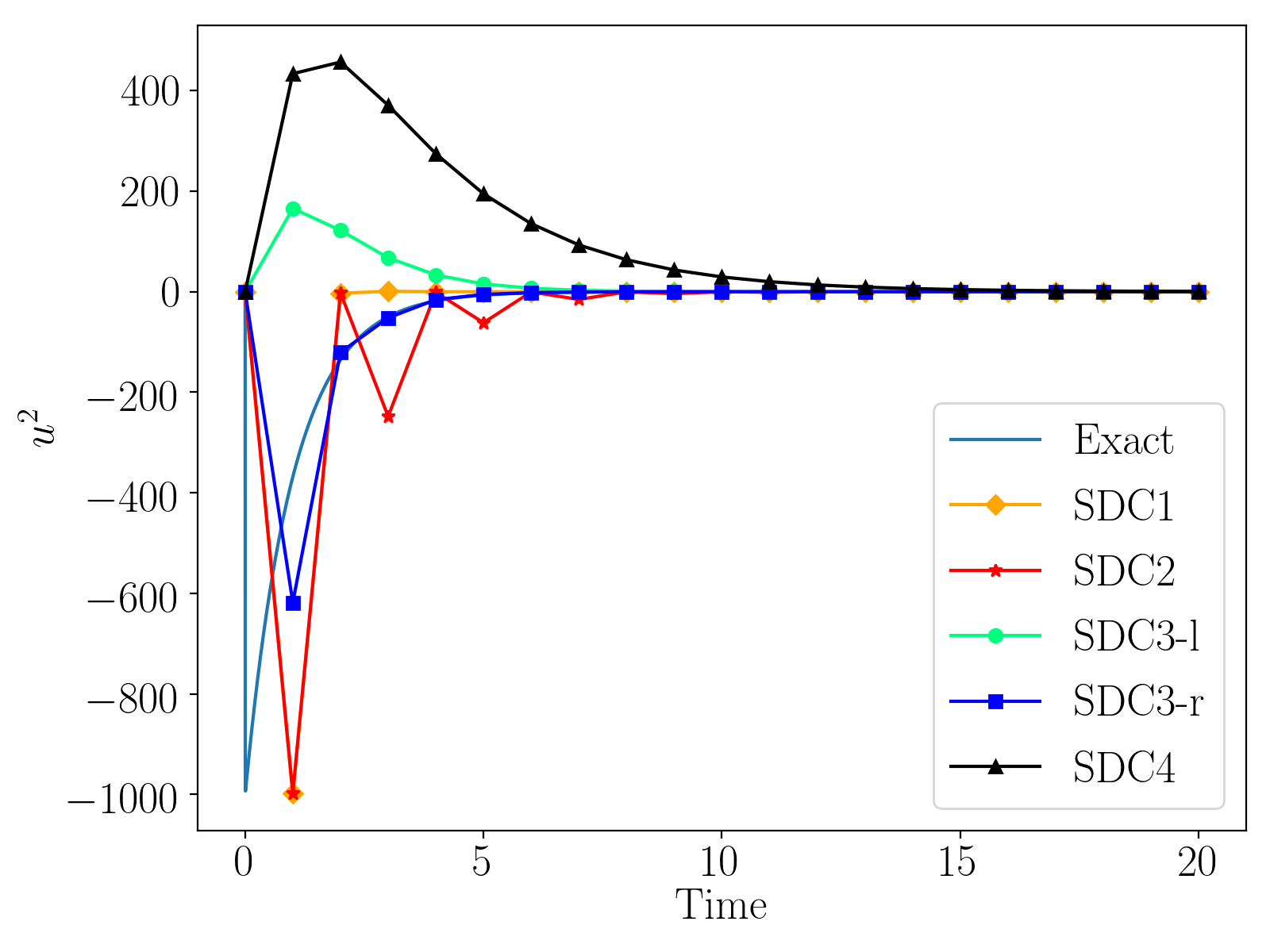}
 \caption{Solutions of the ODE system~($\alpha = 1000$) obtained by partitioned SDC schemes with $\Delta t = 1.0$.}
 \label{FIG: ODE SOL}
\end{figure}

\begin{figure}
  \centering
  \input{ode_3field_GS.tikz}
 \caption{Convergence of the
          SDC1 (\ref{line:ode_3fields:sdc1}),
          SDC2 (\ref{line:ode_3fields:sdc2}),
          SDC3-l (\ref{line:ode_3fields:sdc3-l}),
          SDC3-r (\ref{line:ode_3fields:sdc3-r}), and
          SDC4 (\ref{line:ode_3fields:sdc4}) schemes applied to the ODE system.}
 \label{FIG: ODE 3FIELD}
\end{figure}
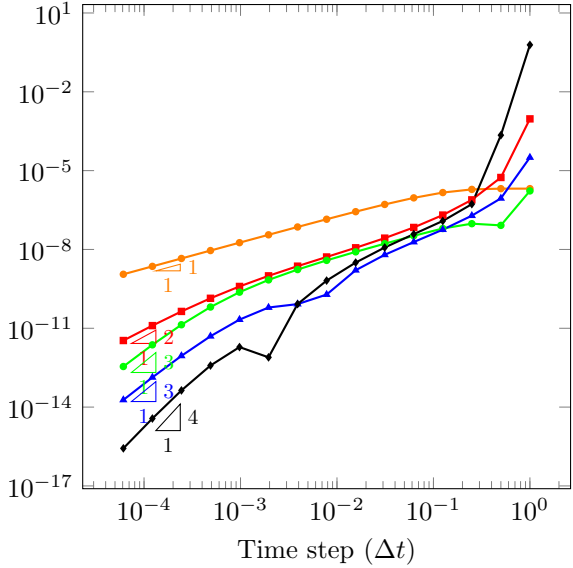
}

\subsection{Advection-diffusion-reaction system}
\label{SEC: APP ADR}
In this section, we consider time-dependent coupled advection-diffusion-reaction
(ADR) systems. These systems have applications in the modeling of chemical reactions
\cite{tezduyar1986discontinuity}, the description for superconductivity of
liquids \cite{estep2000estimating}, and biological predator-prey models
\cite{estep2000using}.
{
In this work, we consider the $2$-dimensional predator-prey model from \cite{estep2000using},
which involves $2$ coupled systems.
The governing equation for the $i$th species is
\begin{equation} \label{EQ: CLAW ADR}
  \pder{u^i}{t} + (v^i \cdot \nabla) u^i - \nabla\cdot( D^i \nabla u^i) =
  f^i(u,\,x,\,t), \quad (x,\,t) \in \Omega \times (0,\,1],
  \quad i = 1, 2.
\end{equation}
Here $u^1$ denotes the prey, $u^2$ denotes the predator, $u = [u^1, u^2]$, $\Omega \subset \Rbb^2$
is the computational domain. $D^i$ is the diffusivity, and $v^i(x) \in \Rbb^2$ is the velocity field for the $i$th species.
The reaction terms are
\begin{equation}
 \begin{aligned}
  f^1(u,\,x,\,t) = u^1(-(u^1 - a^1)(u^1 - 1) - a^2u^2)\,\qquad
  f^2(u,\,x,\,t) = u^2(-a^3 - a^4u^2 + a^2 u^1),
  \label{EQ: PREDATOR PREY SOURCE}
 \end{aligned}
\end{equation}
where $a^1 = 0.25$, $a^2 = 2$, $a^3 = 1$, $a^4 = 3.4$, and the diffusivities
are constant $D^1 = D^2 = 0.01$. The computational domain is the
two-dimensional unit square $\Omega = [-0.5,\,0.5]\times[-0.5,\,0.5]$ with the
prey initially uniformly distributed, and predators initially gathered
near $(x_0,y_0) = (-0.25,-0.25)$
\begin{equation}
 \begin{aligned}
  u^1(x,\,y,\, 0) = 1.0   \quad \textrm{and} \quad  u^2(x,\,y,\,0) =
                \begin{cases}
                  0 & r > d\\
                  e^{-\frac{d^2}{d^2 - r^2}} & r \leq d\\
                \end{cases},
  \end{aligned}
\end{equation}
where $d=0.2$, $r = \sqrt{(x - x_0)^2 + (y-y_0)^2}$. The boundary conditions
are all Neumann conditions $\displaystyle{\pder{u}{n} = 0}$ and the velocity
fields are constant $v^1(x) = (0,0)$ and $v_2(x) = (0.5,0.5)$.
The equations are discretized with a standard high-order discontinuous
Galerkin method using upwind flux for the
inviscid numerical flux and the compact DG flux \cite{peraire2008compact}
for the viscous numerical flux on a $40 \times 40$ structured mesh of
quadratic simplex elements.
}
The governing equations in (\ref{EQ: CLAW ADR}) reduce to the following
system of ODEs after the DG discretization is applied
\begin{equation}
\label{EQ: ADR0}
 \mass[i]\stvcdot[i] = \res[i](\stvc[i]) +
                       \cpl[i](\stvc[1],\,\stvc[2]),
\end{equation}
where $\mass[i]$ is the fixed mass matrix, $\stvc[i](t)$ is the semi-discrete
state vector, i.e., the discretization of $u$ on $\Omega$,
$\res[i](\stvc[i])$ is the spatial discretization of the advection and
diffusion terms on $\Omega$, and $\cpl[i]$ is the coupling term
that contains the DG discretization of the $i$th reaction source term in
(\ref{EQ: PREDATOR PREY SOURCE}).
The solution of (\ref{EQ: ADR0})
using the SDC4 scheme is provided in
\cref{FIG: PREDATOR_PREY} using the time step size $\dt{}=0.1$.
The predators are diffused quickly and migrate diagonally
upward, while the prey are mostly affected by the coupled reaction near the
extent of the predator population.

\begin{figure}[ht]
\centering
 \includegraphics[width=0.3\textwidth]{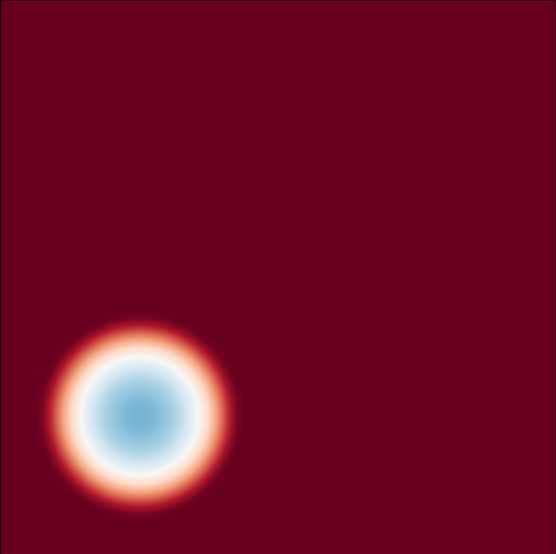} \quad
 \includegraphics[width=0.3\textwidth]{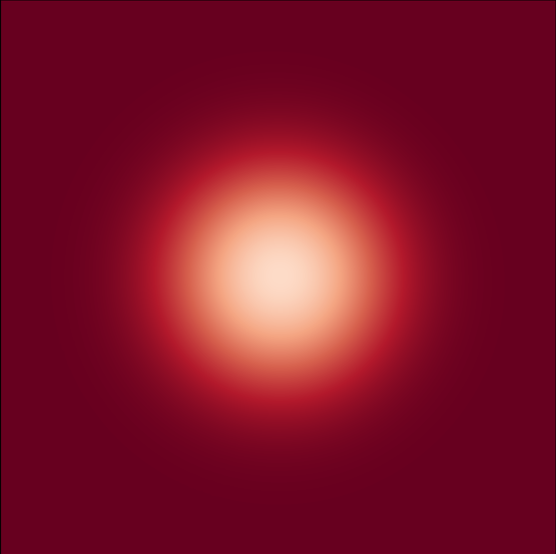} \quad
 \includegraphics[width=0.3\textwidth]{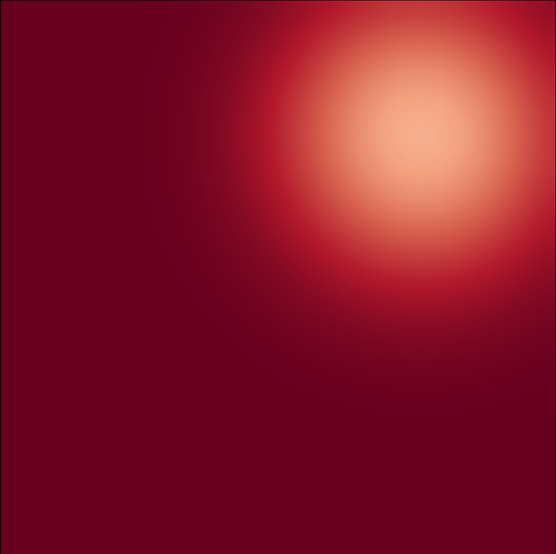} \\\vspace{2mm}
 \includegraphics[width=0.3\textwidth]{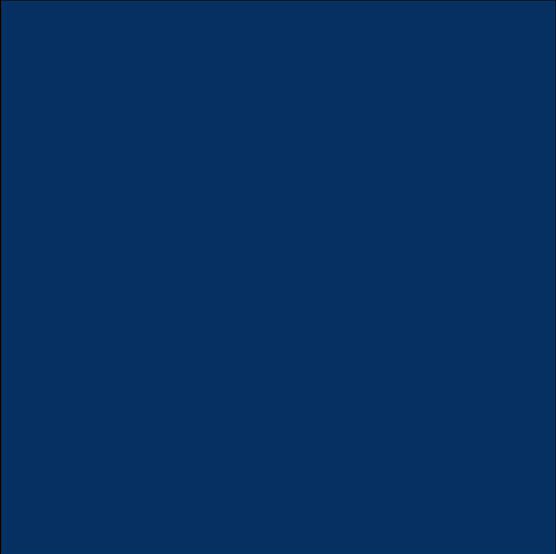} \quad
 \includegraphics[width=0.3\textwidth]{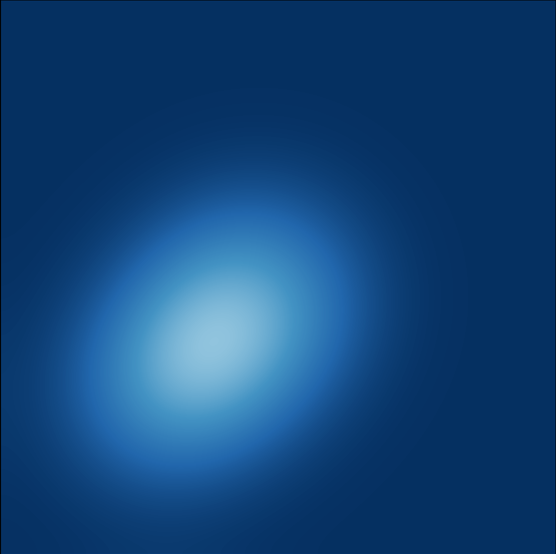} \quad
 \includegraphics[width=0.3\textwidth]{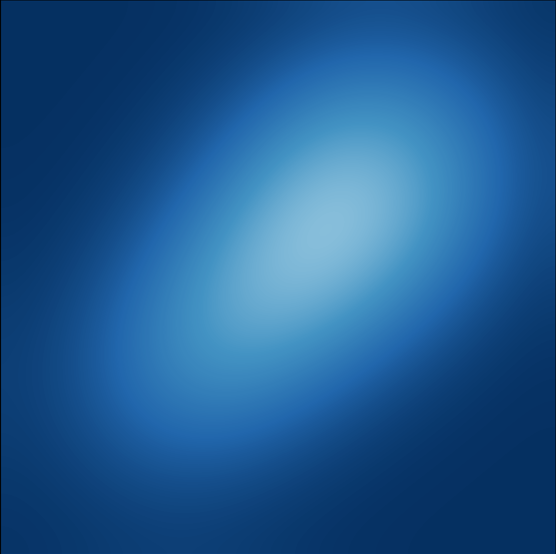}
 \caption{Predator (\emph{top}) and prey (\emph{bottom}) populations at
          various snapshots in time: $t = 0.0$ (\emph{left}), $t = 0.5$
          (\emph{center}), and $t = 1.0$ (\emph{right}).}
 \label{FIG: PREDATOR_PREY}
\end{figure}

To validate the temporal convergence of the high-order partitioned scheme, we apply
these SDC schemes introduced in \cref{sec:sdc_schemes} and, similar to the previous section,
we use the $L_{\infty}$-error between a reference solution and the numerical
solution at time $t = 1.0$
to quantify the error
where the reference solution
at $t = 1.0$ obtained by using the SDC4 scheme with
$\dt{} = 6.25 \times 10^{-3}$.
The errors as a function of the time step size are provided in
\cref{FIG:PREDATOR_PREY_ERR}, which verifies the design order
of accuracy of all SDC schemes. This
figure also shows that no stability issues were observed for any of the
results, even for the coarsest time step $\Delta t = 0.1$.
\begin{figure}
  \centering
  \input{predator_prey_GS.tikz}
 \caption{Convergence of the
          SDC1 (\ref{line:predator_prey:sdc1}),
          SDC2 (\ref{line:predator_prey:sdc2}),
          SDC3-l (\ref{line:predator_prey:sdc3-l}),
          SDC3-r (\ref{line:predator_prey:sdc3-r}), and
          SDC4 (\ref{line:predator_prey:sdc4}) schemes applied to the predator-prey model.}
 \label{FIG:PREDATOR_PREY_ERR}
\end{figure}
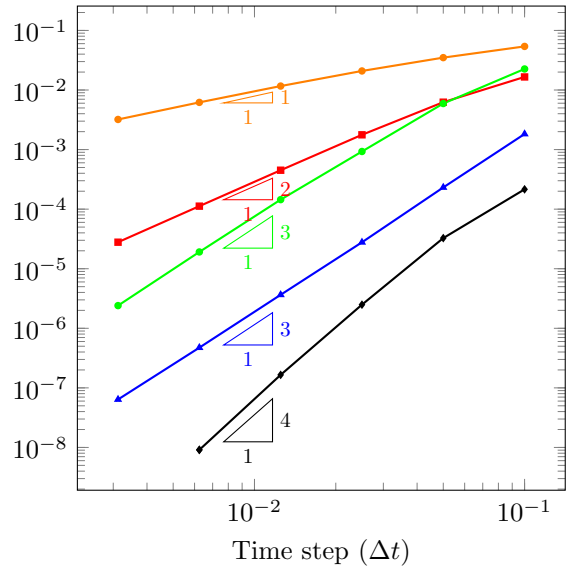

\subsection{Modified cavity problem}\label{sec: modified cavity}
In this section, we study the modified driven cavity problem with flexible
bottom\footnote{The detailed implementation of the modified cavity problem is in \url{https://zhengyu_huang@bitbucket.org/zhengyu_huang/incompressible_fsi_2d.git}.} depicted in \cref{fig:Cavity}. This problem was first introduced in~\cite{wall1999fluid} and since then
has been used as a benchmark problem for a variety of FSI
studies~\cite{forster2007artificial,kuttler2008fixed, gerbeau2003quasi,kassiotis2011nonlinear,habchi2013partitioned}. An
oscillating velocity $\bm{v}(t) = (1 - \cos(2 \pi t/5), 0)$ is imposed on the
top of the cavity.  Each side is of length $1$ containing three elements (two
unconstrained nodes) that allow free inflow and outflow of
fluid, i.e.\ homogeneous Neumann boundary conditions are imposed on these
apertures. This way the structural displacements are not constrained by the
fluid's incompressibility~\cite{kuttler2006solution}. The fluid density and
dynamic viscosity are $\rho^f = 1$ and $\mu^f = 0.01$. The structure is of
thickness $h = 0.002$ and Young's modulus $E = 250$. The density of the structure varies for
different test cases to demonstrate the stability of the coupling procedures.
Decreasing structure density increases difficulties to the coupling algorithm
since the main resistance of the structure against the fluid pressure stems from
its mass.
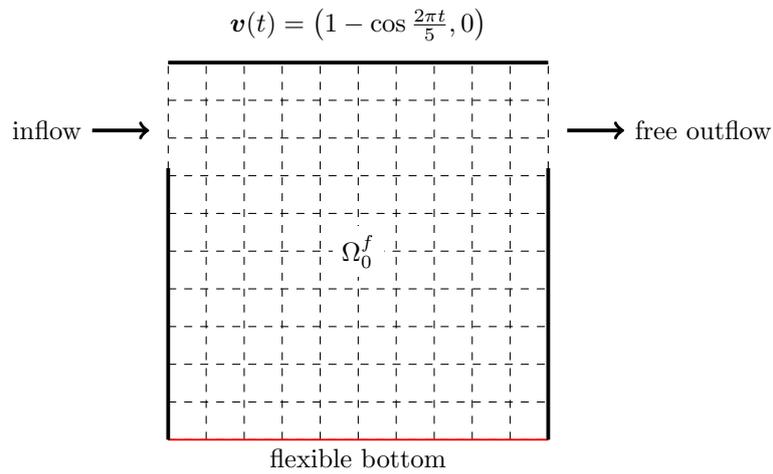
\begin{figure}
\centering
    \begin{tikzpicture}[scale=0.5]
      \foreach \i in {0,...,10}
      {
      \draw [dashed] (\i, 0)-- (\i,10);
      \draw [dashed] (0, \i)-- (10, \i);
      }
      \draw [line width=0.25mm, red] (0, 0)-- (10,0);
      \draw [line width=0.5mm] (0, 10)-- (10,10);
      \draw [line width=0.5mm] (0, 0)-- (0,7.2);
      \draw [line width=0.5mm] (10, 0)-- (10,7.2);

      \node[scale=1] at (5, -0.5) {\text{flexible bottom}};
      \draw[thick,->,line width=0.5mm,](-2.0,8.2) node[anchor=east,align=right][scale=1] {inflow} -- (-0.5,8.2) ;
      \draw[thick,->,line width=0.5mm,](10.5,8.2) -- (12.0,8.2) node[anchor=west][scale=1] {free outflow};
      \node[scale=1] at (5.0, 11.0) {$\bm{v}(t)=\big(1 - \cos\frac{2 \pi t}{5}, 0\big)$};

      \node[scale=1,fill=white] at (5,5) {$\Omega_0^f$};

    \end{tikzpicture}
\caption{Schematics of the modified cavity problem.}
\label{fig:Cavity}
\end{figure}

The considered Newtonian fluid is governed by the incompressible Navier-Stokes equations~\cite{forster2006geometric}, written on the undeformed fluid domain $\Omega^{f}_0$,
\begin{equation}
 \begin{aligned}
\frac{\partial \bm{v}}{\partial t} + ( \bm{v} - \dot{\bm{d}}^x) : \nabla_x \bm{v} - \nabla_x : \frac{\bm{\sigma}^f}{\rho^f} &= 0 \qquad&&\textrm{in}~\Omega_0^f, \\
\nabla_x  \cdot  \bm{v} &= 0 \qquad&&\textrm{in}~\Omega^f_0,\\
\bm{v}(t) &= \bm{v}_D(t) \qquad&&\textrm{in}~\Gamma^f_{0D}\\
\bm{\sigma}^f \cdot \bm{N} &= 0 &\qquad&\textrm{in}~\Gamma^f_{0N}.\\
 \end{aligned}
\label{eq:incompressible_ns}
\end{equation}
Here $\bm{v}$ denotes the fluid velocity field, $\dot{\bm{d}}^x$ denotes the mesh velocity, $\bm{\sigma}^f = -\mu^f(\nabla_x\bm{u} + \nabla_x\bm{u}^T) + p\mathbb{I}$ denotes the stress tensor,
and $p$ denotes the pressure field.  The time derivative in \cref{eq:incompressible_ns} describes the temporal change of velocity on a reference point while all spatial derivatives refer to the deformed domain. Dirichlet boundaries are imposed on both top and bottom, and two side walls denoted as $\Gamma^f_{0D}$, and homogeneous Neumann boundary conditions are imposed on both apertures, with outward normal $\bm{N}$ and ambient pressure $0$.

The governing equations in (\ref{eq:incompressible_ns}) is semi-discretized by $32 \times 32$ traditional Taylor-Hood $Q_2$-$Q_1$ mixed elements,
i.e.\ continuous biquadratic velocity and continuous bilinear pressure, which satisfies the Babu\v{s}ka-Brezzi condition~\cite[p.~286]{donea2003finite}.
This leads to the following system of ODEs,
\begin{equation}
\label{EQ: FSI3 FLUID COUPLING}
 \mass[f]\stvcdot[f] = \res[f](\stvc[f],\,\cpl[f]),
\end{equation}
where $\mass[f]$ is the fixed mass matrix, $\stvc[f](t)$ is the semi-discrete
fluid state vector, i.e.\ the discretization of $\bm{v}$ and $p$ on $\Omega^f(t)$,
$\res[f](\stvc[f],\,\cpl[f])$ denotes the spatial discretization of the
inviscid and viscous fluxes on $\Omega^f_0$, and $\cpl[f]$ is the coupling term
that contains information about the mesh position $\bm{d}^{x}$ and velocities $\dot{\bm{d}}^{x}$.
It is worth mentioning that the mass matrix $\mass[f]$ is singular due to the incompressibility constraints,
which causes the added mass effect instability \cite{causin2005added, forster2007artificial,van2009added} for incompressible flows.

The flexible bottom, with fixed ends, is modeled as the nonlinear beam written in the Lagrangian form in the undeformed domain~$\Omega_0^s$.
The equilibrium equation of the beam in the weak form can be written as~\cite[p.~308]{de2012nonlinear}
\begin{equation}
\int_{\Omega^s_0} \rho^s \ddot{\bm{d}}^s \delta \bm{d}^s dX  + \int_{\Omega^s_0} \sigma_{11}^{s} \delta \epsilon_{11}^{s}dX = \bm{f}_{ext}\delta\bm{d}^s.
\end{equation}
Here the first term is the contribution of the inertial forces, the second and third terms represent the virtual work of the internal forces and the external forces. $\sigma_{11}^{s}$ and $\epsilon_{11}^{s}$ are nonlinear axial strain and axial stress components, related by the linear elasticity constitutive relation.

The flexible bottom is discretized by $32$ beam elements with linear shape functions for the horizontal displacement
and cubic Hermitian shape functions for the vertical displacement.
The discretized equation becomes
\begin{equation}
\mass[s]\stvcdot[s] = \res[s](\stvc[s],\,\cpl[s]),
\end{equation}
where $\mass[s]$ is the fixed mass matrix, $\stvc[s](t)$ is the semi-discrete
state vector consisting of the displacements and velocities of the
beam nodes, $\res[s](\stvc[s],\,\cpl[s])$ is the spatial discretization
of the virtual work and boundary conditions on the reference domain
$\Omega^s_0$, and $\cpl[s]$ is the coupling term that contains information about
the flow load on the structure.

{ To determine the deformation of the fluid mesh, the mesh is} considered as a linear pseudo-structure~\cite{farhat1995mixed, farhat1998torsional} driven solely
by Dirichlet boundary conditions provided by the displacement of the structure
at the fluid-structure interface. The governing equations are given by the
continuum mechanics equations in the Lagrangian form with the linear elastic constitutive relation in the undeformed fluid
domain $\Omega^x_0$,
\begin{equation}
 \begin{aligned}
  \rho^{x}\ddot{\bm{d}}^x - \nabla_X : \bm{\sigma}^x &= 0
    \qquad\qquad &&\text{in}~\Omega^x_0, \\
  \bm{d}^x &= \bm{d}_D(t)
    \qquad\qquad &&\text{on}~\partial \Omega^x_{0D}, \\
  \dot{\bm{d}}^x &= \dot{\bm{d}}_D(t)
    \qquad\qquad &&\text{on}~\partial \Omega^x_{0D},
 \end{aligned}
\label{EQ: FSI3 MESH}
\end{equation}
where $\rho^x =500$ is the density, and $\bm{\sigma}^{x}$ is the Cauchy stress tensor.
The position and velocity of the fluid domain are
prescribed along $\partial \Omega^x_{0D}$, the union of the fluid-structure
interface and the fluid domain boundary.

The governing equations given by (\ref{EQ: FSI3 MESH}) are discretized with $32
\times 32$ biquadratic elements and reduced to the following system of ODEs,
\begin{equation}
 \Mbm^x\dot\ubm^x = \rbm^x(\ubm^x,\,\cbm^x),
\end{equation}
where $\mass[x]$ is the fixed mass matrix, $\stvc[x]$ is the semi-discrete
state vector consisting of the displacements and velocities of the
mesh nodes, $\res[m](\stvc[x],\,\cpl[x])$ is the spatial discretization
of the continuum equations and boundary conditions on the reference domain
$\Omega^x$, and $\cpl[x]$ is the coupling term that contains information about
the motion of the fluid structure interface.

Finally, we obtain the three-field coupled fluid-structure equations
\begin{equation}
\label{EQ: FSI3 SEMI DISC}
\mass[s]\stvcdot[s] = \res[s](\stvc[s],\,\cpl[s]), \quad
\mass[x]\stvcdot[x] = \res[x](\stvc[x],\,\cpl[x]) , \quad
\mass[f]\stvcdot[f] = \res[f](\stvc[f],\,\cpl[f]).
\end{equation}
The coupling terms have the following dependencies
\begin{equation}
\cpl[s] = \cpl[s](\stvc[s],\,\stvc[x],\,\stvc[f]), \quad
\cpl[x] = \cpl[x](\stvc[s]), \quad
\cpl[f] = \cpl[f](\stvc[s],\, \stvc[x]).
\label{EQ: FSI3 COUPLE}
\end{equation}
The ordering of the subsystems implied in
\cref{EQ: FSI3 SEMI DISC} is used throughout
the remainder of this section, which plays an important role when
defining the Gauss-Seidel predictors---only a single predictor $\tilde{\bm c}^s$ is needed to decouple the multiphysics system.
The conservative load and motion transfer algorithms~\cite{farhat1998load} are applied to evaluate these coupling terms.

\begin{figure}
 \includegraphics[width=0.3\textwidth]{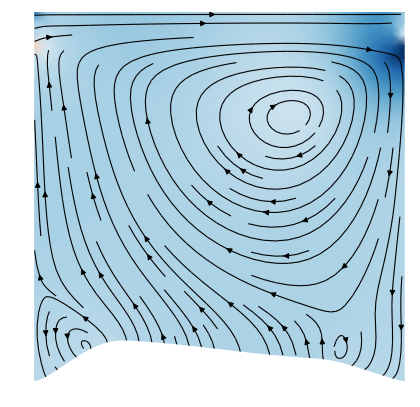} \quad
 \includegraphics[width=0.3\textwidth]{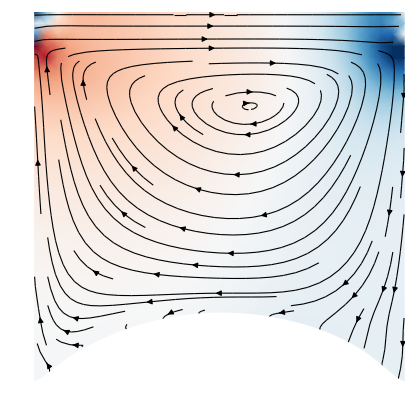} \quad
 \includegraphics[width=0.3\textwidth]{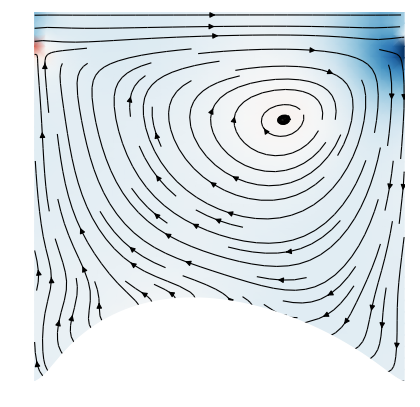} \\
 \includegraphics[width=0.3\textwidth]{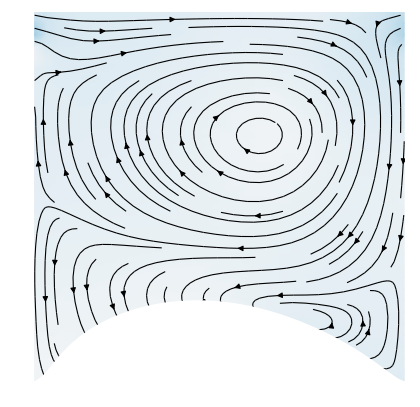} \quad
 \includegraphics[width=0.3\textwidth]{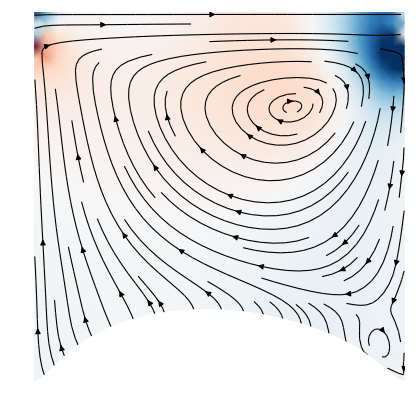} \quad
 \includegraphics[width=0.3\textwidth]{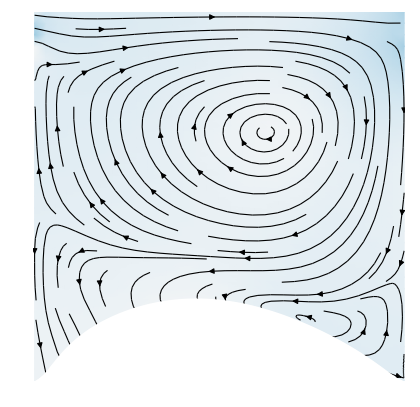}
 \caption{Driven cavity pressure field and velocity streamlines t=$4$, $22$, $44$, $65$, $83$, $100$
          (\emph{left}-to-\emph{right}, \emph{top}-to-\emph{bottom}).}
 \label{fig: cavity pressure streamlines}
\end{figure}

The aforementioned SDC solvers are
applied to this benchmark problem,
 The time step is fixed to be $0.1$, the simulation time is $T = 100$~($20$ periods),
 and the Young's modulus and Poisson's ratio of the pseudo-structure are set to be $250$ and $0.0$. The flow is initially quiescent with
 0 pressure, the same as the ambient pressure.
Snapshots of the pressure field and the streamlines are shown in~\cref{fig: cavity pressure streamlines}, with the structure density $\rho^s = 500$.
The flexible bottom undergoes large deformations and oscillates along with the prescribed periodic velocity at the top.
To understand the stability of the proposed partitioned solvers, we vary the density of the structure $\rho^s$ by multiples of one hundred.
The minimal structure density that leads to a stable simulation for different SDC 
schemes are reported in \cref{tab:cavityflow_density}.  The corresponding vertical displacements of the central point on the flexible bottom are depicted in
\cref{fig: cavity disp}, no spurious oscillations are observed, which indicates numerical stability.
Here the backward Euler scheme~(BE) is from \cite{forster2007artificial},
which solves the incompressible flow by using the backward Euler scheme, and the structure with generalized-$\alpha$ time integration scheme~\cite{chung1993time,forster2007artificial}. Its sequentially staggered algorithm~{(no iterations)} is equipped with a first order structure displacement predictor,
which is the most stable partitioned solver reported in \cite{forster2007artificial}.  Its minimal stable structure density in the present setup is $900$, which outperforms the
SDC1 scheme, thanks to the improved numerical dissipation from the generalized-$\alpha$ time integration scheme.
However, SDC2 and SDC3-r schemes are stable with $\rho^s = 500$ and  $\rho^s = 400$, which demonstrates the superior stability of the proposed high order partitioning solvers. And SDC3-r scheme is the most stable scheme for this test case, which demonstrates the possibility to improve stability through judiciously choosing the low order approximation $\mathcal{C}_j$.
Moreover,  it is worth mentioning that IMEX based high order partitioned solvers~\cite{Huang2019} can not
handle this case, due to the singular fluid mass matrix; more comparisons will be presented in \cref{sec: foil damper}.

\begin{table}
\centering
\begin{tabular}{c|cccccc}
\toprule[1.5pt]
Method    & BE & SDC1 &  SDC2 & SDC3-l & SDC3-r & SDC4\\
\hline
$\rho^s$   & 900 & 1200 &  500 & 800  &  400  & 1000    \\
\bottomrule[1.5pt]
\end{tabular}
\caption{Minimum structure density to maintain the stability of the modified cavity flow problem for different schemes.}\label{tab:cavityflow_density}
\end{table}

\begin{figure}
\centering
 \includegraphics[width=0.72\textwidth]{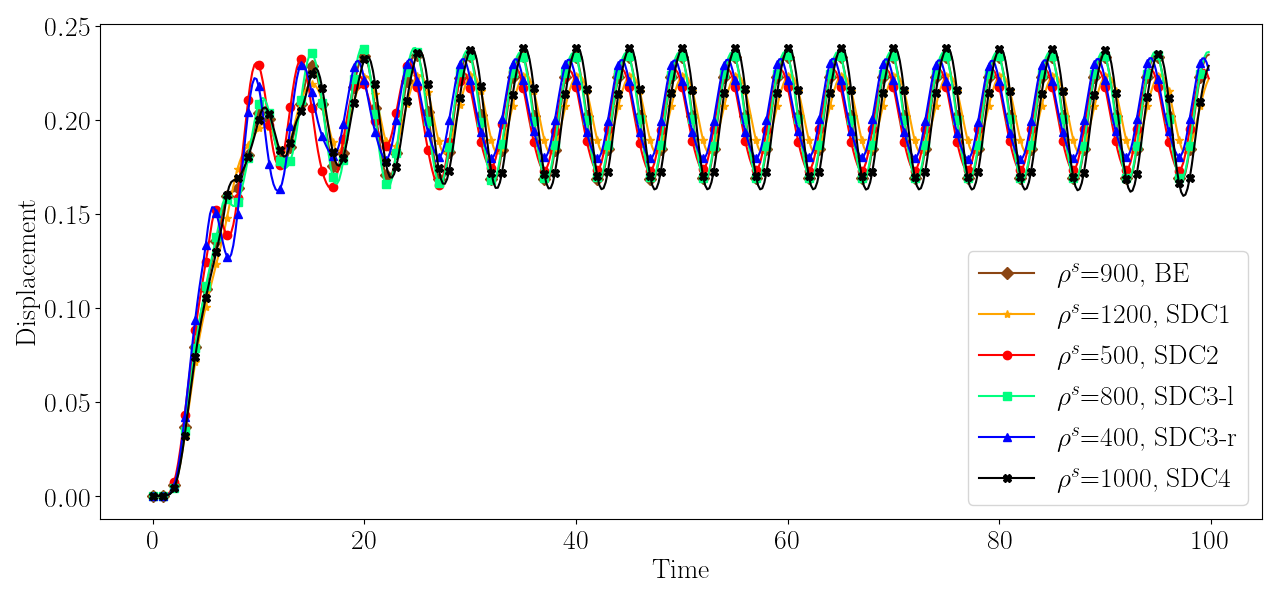} \quad
 \caption{Displacement of the central point at the flexible bottom of the cavity for different schemes at the minimum stable density.}
 \label{fig: cavity disp}
\end{figure}

\subsection{Foil damper problem}
\label{sec: foil damper}
In this section, we demonstrate the proposed SDC solvers on the energy-harvesting model problem \cite{peng2009energy, zahr2016adjoint, Huang2019}.
Consider the foil-damper system
in~\cref{FIG: FOIL DAMPER} suspended in an isentropic, viscous
flow where the rotational motion is a prescribed periodic motion
$\theta(t) = \frac{\pi}{4}\cos(2\pi f t)$ with frequency $f = 0.2$ and
the vertical displacement is determined by balancing the forces exerted
on the airfoil by fluid and damper.

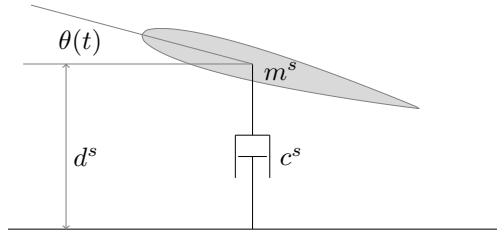
\begin{figure}
 \centering
 \input{foil-with-damper.tikz}
 \caption{Schematics of the foil-damper system.} \label{FIG: FOIL DAMPER}
\end{figure}

The considered Newtonian fluid is governed by the compressible Navier-Stokes equations, defined on a
deformable fluid domain $\Omega^f(t)$, which can be written as a viscous conservation law
\begin{equation}
\label{EQ: FSI GOVERN}
\frac{\partial U}{\partial t} + \nabla_x \cdot \Fcal^{inv}(U) + \nabla_x \cdot \Fcal^{vis}(U, \nabla_x U)= 0 \quad \text{in} \quad \Omega^f(t),
\end{equation}
where $U$ is the conservative state variable vector and the physical
flux consists of an inviscid part $\Fcal^{inv}(U)$ and a viscous part
$\Fcal^{vis}(U,\,\nabla_x U)$,

\begin{equation}
U =
\begin{bmatrix}
\rho^f   \\
\rho^f \bm{v} \\
E
\end{bmatrix},\,
\Fcal^{inv}(U) =
\begin{bmatrix}
\rho^f\bm{v} \\
\rho^f \bm{v}\otimes\bm{v} + p\mathbb{I}\\
(E + p)\bm{v}
\end{bmatrix},\,\textrm{ and }
\Fcal^{vis}(U, \nabla U) =
\begin{bmatrix}
0   \\
\bm{\tau} \\
\bm{\tau}\cdot\bm{v} - \bm{q}
\end{bmatrix},\,
\end{equation}
here $\rho^f$ is the fluid density, $\bm{v}$ is the velocity, and $E$ is the total energy per unit volume.
The stress tensor and the heat flux are given by
\begin{equation}
\bm{\tau} = -\frac{2}{3}\mu^f(\nabla \cdot \bm{v})\mathbb{I} + \mu^f(\nabla \bm{v} + \nabla\bm{v}^T), \quad
\bm{q} = -\kappa \nabla T,
\end{equation}
where $\mu^f$ is the dynamic viscosity, and $\kappa$ is the thermal conductivity, and $T$ is the temperature.
The isentropic assumption states
the entropy of the system is assumed constant, which is tantamount to the flow
being adiabatic and reversible. For a perfect gas, the entropy is defined as
\begin{equation}\label{eqn:entropy}
  s = p/(\rho^f)^\gamma,
\end{equation}
here $\gamma$ is the specific heat ratio.

The conservation law in (\ref{EQ: FSI GOVERN})
is transformed to a fixed reference domain $\Omega^f_0$ by defining a
time-dependent diffeomorphism $\Gcal$ between the reference domain and
physical domain; see~\cref{FIG: DOM MAP}. At each time $t$, a point
$X$ in the reference domain $\Omega^f_0$ is mapped to $x(X,t) = \Gcal(X,t)$
in the physical domain $\Omega^f(t)$.
\begin{figure}
  \centering
  \includegraphics[width=2.5in]{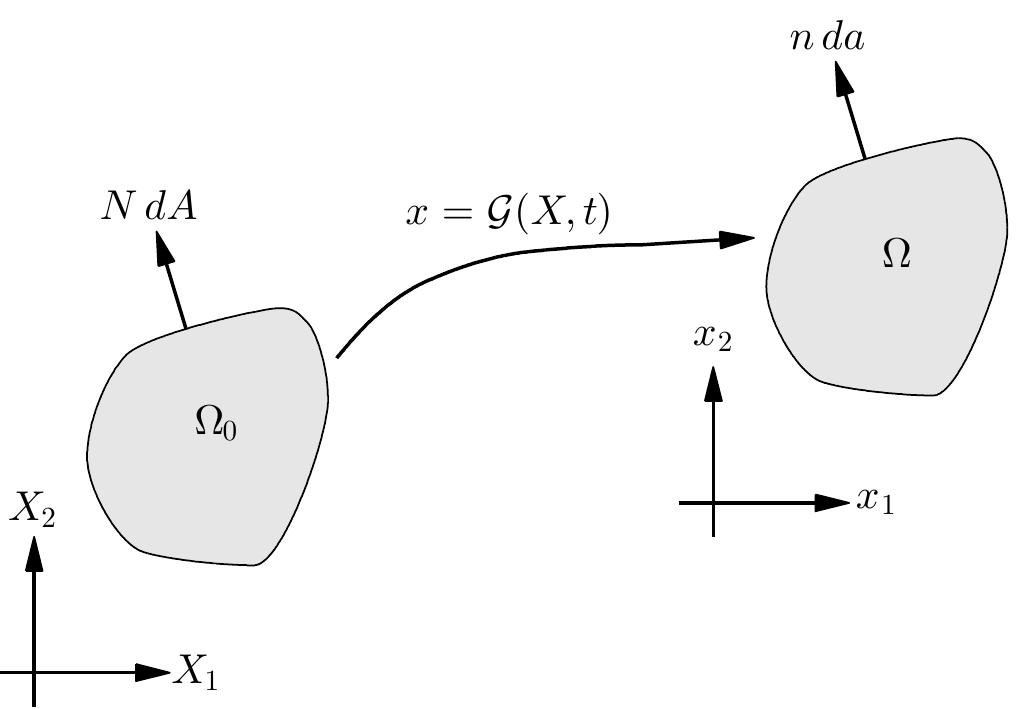}
  \caption{Mapping between reference and physical domains.}
  \label{FIG: DOM MAP}
\end{figure}
The deformation gradient $G$, velocity $v_G$, and Jacobian $g$ of
the mapping are defined as
\begin{equation}
 G = \nabla_X\Gcal\,,\quad
 v_G = \pder{\Gcal}{t}\,,\quad
 g = \det G.
\end{equation}
Following the procedure in \cite{persson2009discontinuous, zahr2016adjoint},
the governing equation
(\ref{EQ: FSI GOVERN}) can be written in the reference domain as
\begin{equation}\label{EQ: FSI GOVERN TRANSF}
\frac{\partial U_X}{\partial t} + \nabla_X \cdot \Fcal_X^{inv}(U_X) +
\nabla_X\cdot \Fcal_X^{vis}(U_X, \nabla_X U_X)= 0
\quad \text{in} \quad \Omega_0^f,
\end{equation}
where $\nabla_X$ defines the spatial derivative with respect to the reference
domain, conserved quantities and its derivatives in the reference domain are
written as
\begin{equation}
 U_X = gU\,,\qquad
 \nabla_X U_X = g\nabla_x U \cdot G + g^{-1}U_X \frac{\partial g}{\partial X}.
\end{equation}
The inviscid and viscous fluxes are transformed to the reference domain as
\begin{equation}
\begin{aligned}
 \Fcal_X^{inv}(U_X) &= g\Fcal^{inv}(g^{-1} U_X)G^{-T} -
                       gU_X \mathbin{\mathop{\otimes}} v_G G^{-T}, \\
 \Fcal_X^{vis}(U_X) &= g\Fcal^{vis}\left(g^{-1} U_X,
                       g^{-1}\left[\nabla_X U_X -
                       g^{-1} U_X \pder{g}{X}\right]G^{-1}\right)G^{-T}.
 \end{aligned}
\end{equation}

The governing equations given by (\ref{EQ: FSI GOVERN TRANSF}) are discretized with a standard high-order discontinuous
Galerkin method using Roe's flux \cite{roe1981approximate} for the inviscid
numerical flux and the compact DG flux \cite{peraire2008compact} for the
viscous numerical flux. The DG discretization uses a mesh consisting of $3912$
cubic simplex elements, and leads to the following
system of ODEs
\begin{equation}
\label{EQ: FSI3 FLUID COUPLING COMP}
 \mass[f]\stvcdot[f] = \res[f](\stvc[f],\,\cpl[f]),
\end{equation}
where $\mass[f]$ is the fixed mass matrix, $\stvc[f](t)$ is the semi-discrete
fluid state vector, i.e., the discretization of $U_X$ on $\Omega_0^f$,
$\res[f](\stvc[f],\,\cpl[f])$ is the spatial discretization of the transformed
inviscid and viscous fluxes on $\Omega^f_0$, and $\cpl[f]$ is the coupling term
that contains information about the domain mapping $\Gcal(X,\,t)$. In
particular, the coupling term contains the position and velocities of the
nodal coordinates of the computational mesh.
{{The domain mapping is defined using a nodal (Lagrangian) basis.}}

The foil is modeled as a simple mass-spring-damper system
that can directly be written as a second-order system of ODEs, with respect to the vertical displacement $d^s$, as follows,
\begin{equation} \label{EQ: SIMP STRUCT}
 m^s\ddot{d}^s + c^s\dot{d}^s + k^s d^s = f_{ext}(t),
\end{equation}
where $m^s$ is the mass of the (rigid) object, $c^s = 1$ is the damper resistance
constant, $k^s = 0$ is the spring stiffness, and $f_{ext}(t)$ is a time-dependent
external load, which will be given by integrating the pointwise force the
fluid exerts on the object. This simple structure allows us to study the
stability and accuracy properties of the proposed high-order partitioned
solver for this class of multiphysics problems without the distraction of
transferring solution fields across the fluid-structure interface.

To conform to the notation in this document and encapsulate the
semi-discretization of PDE-based structure models, the equation in
(\ref{EQ: SIMP STRUCT}) is re-written in a first-order form as
\begin{equation}
 \mass[s]\stvcdot[s] = \res[s](\stvc[s],\,\cpl[s]).
\end{equation}
In the case of the simple structure in (\ref{EQ: SIMP STRUCT}), the mass
matrix, state vector, residual, and coupling term are
\begin{equation}
 \mass[s] = \begin{bmatrix} m^s & \\ & 1 \end{bmatrix}, \qquad
 \stvc[s] = \begin{bmatrix} \dot{d}^s \\ d^s \end{bmatrix}, \qquad
 \cpl[s] = f_{ext}, \qquad
 \res[s](\stvc[s],\,\cpl[s]) =
              \begin{bmatrix} f_{ext}-c^s\dot{d}^s-k^s d^s \\ \dot{d}^s \end{bmatrix}.
\label{EQ: FSI3 STRUCT COUPLING}
\end{equation}

The motion of the fluid mesh is described as a blending map \cite{persson2009discontinuous}. That is, the domain mapping
$x = \Gcal(X,\,t)$ is given by an analytical function, parametrized by the
deformation and velocity of the fluid-structure interface, that can be
analytically differentiated to obtain the deformation gradient $G(X,\,t)$
and velocity $v_G(X,\,t)$. Since the fluid mesh motion is no longer included
in the system of time-dependent partial differential equations, this leads to
a two-field FSI formulation in terms of the fluid and structure states only.

In the two-field FSI setting
\begin{equation}
\label{EQ: FSI2 SEMI DISC}
\mass[s]\stvcdot[s] = \res[s](\stvc[s],\,\cpl[s]), \quad
\mass[f]\stvcdot[f] = \res[f](\stvc[f],\,\cpl[f]),
\end{equation}
the mesh motion is given by an analytical function and the coupling terms have
the following dependencies
\begin{equation} \label{EQ: FSI2 COUPLE}
\cpl[s] = \cpl[s](\stvc[s],\,\stvc[f]), \quad
\cpl[f] = \cpl[f](\stvc[s]).
\end{equation}
In this case, the structure coupling term is determined from the fluid and
structure state since the external force depends on the traction integrated
over the fluid-structure interface. The fluid coupling term, i.e., the
position and velocity of the fluid mesh, is determined from the structure
state. Finally, the ordering is chosen as in~\cref{sec: modified cavity}, and only a single predictor $\tilde{\bm c}^s$ is needed to
decouple the multiphysics system.

Snapshots of the vorticity
field and motion of the airfoil are shown in~\cref{FIG: FOIL DAMPER VORT}
for a single configuration of the fluid-structure system.
\begin{figure}
 \includegraphics[width=0.3\textwidth]{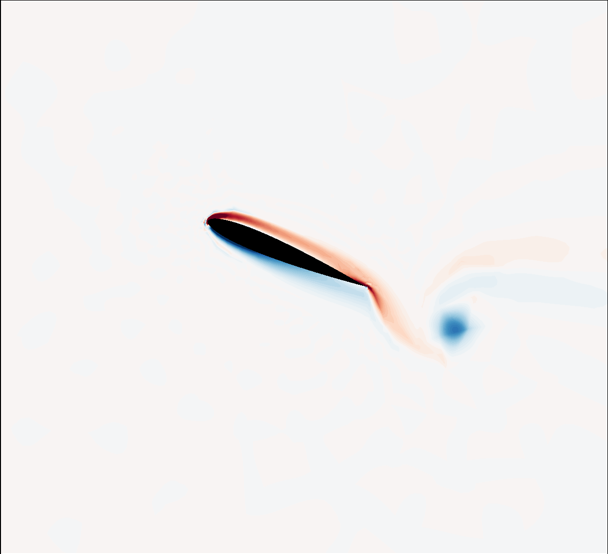} \quad
 \includegraphics[width=0.3\textwidth]{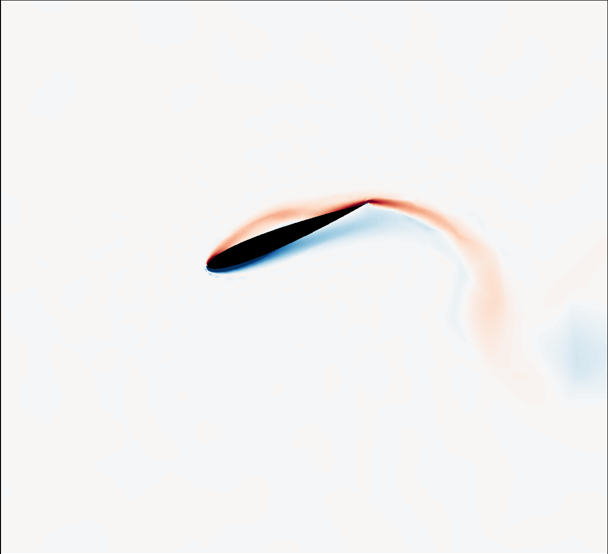} \quad
 \includegraphics[width=0.3\textwidth]{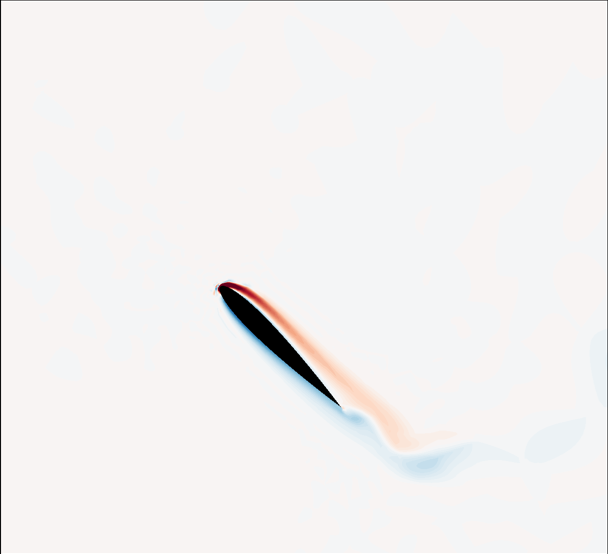} \\
 \includegraphics[width=0.3\textwidth]{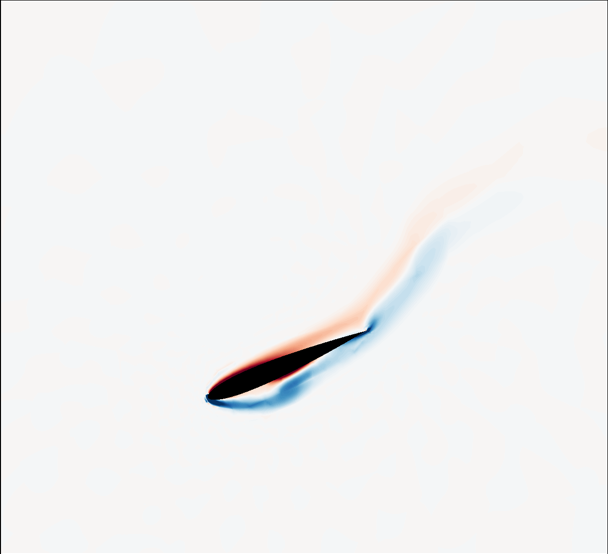} \quad
 \includegraphics[width=0.3\textwidth]{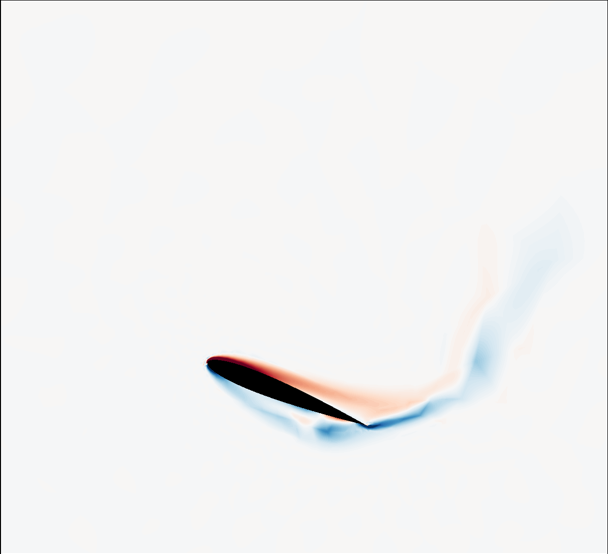} \quad
 \includegraphics[width=0.3\textwidth]{{nacamsh1ref0p3_Re1000_Mridx0_Nrefdt02_SDC4_frame0200.vort}.png}
 \caption{Airfoil motion and flow vorticity corresponding to foil-damper
          system under prescribed rotational motion
          $\theta(t) = \frac{\pi}{4}\cos(2\pi f t)$ with frequency $f = 0.2$
          at various snapshots in time:
          $t = 0.83,\,1.67,\,2.5,\,3.33,\,4.17,\,5.0$
          (\emph{left}-to-\emph{right}, \emph{top}-to-\emph{bottom}) with $m^s$ = 1, $\Delta t$ = 0.025, and SDC4 scheme.}
 \label{FIG: FOIL DAMPER VORT}
\end{figure}
Our numerical experiments study the stability of the proposed SDC partitioned schemes as a function of the mass ratio between the structure and fluid,
an important parameter that can impact the stability of partitioned solvers
as identified in \cite{causin2005added, van2007higher}, and the time step size for
SDC schemes up to fourth order. The mass ratio, $\bar{m}$, considered is the
ratio of the mass of the structure, $m^s$, to the mass of fluid displaced by the
structure, $\rho^f A$, where $A = 0.08221$ is the area of the airfoil. Since the isentropic Navier-Stokes
equations can be seen as an artificial compressibility formulation for
the incompressible Navier-Stokes equations
\cite{lin1995incompressible, desjardins1999incompressible}, we consider the
reference fluid density to be constant and equal to the freestream density $\rho^f = 1$. Variations in
the mass ratio are achieved by varying the mass of the structure with all
other parameters fixed. The stability results are summarized in
~\cref{FIG: FOIL DAMPER STAB} where ~\ref{line:foil_damper_stable}~
indicates a $(\Delta t,\,\bar{m})$-pair that leads to a stable simulation and
~\ref{line:foil_damper_unstable}~ leads to an unstable one. The corresponding
foil vertical displacements for these blue dots adjacent to these unstable red dots
are depicted in~\cref{FIG: FOIL DAMPER DISP}; no trail of unstable oscillation appears.

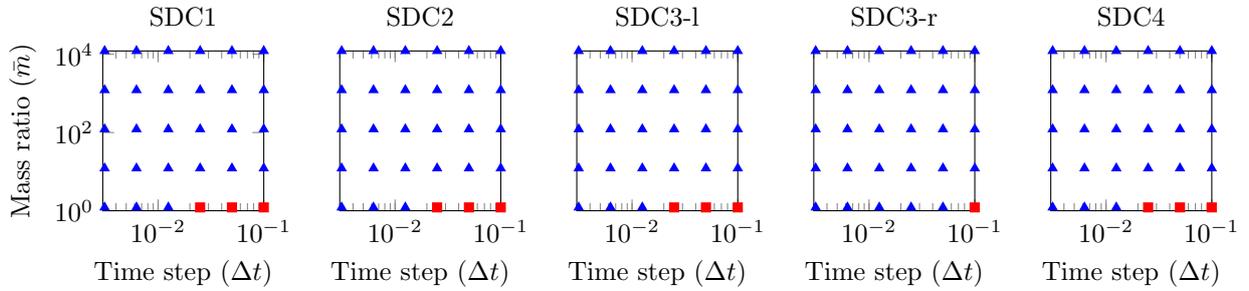
\begin{figure}
 \centering
 \input{foil_damper_stability.tikz}
 \caption{Behavior of the predictor-based partitioned schemes for a range
          of mass ratios and time steps for SDC1, SDC2, SDC3-l, SDC3-r, and SDC4 (\emph{left} to
          \emph{right}) schemes. \emph{Legend}: \ref{line:foil_damper_stable}
          indicates a stable simulation and \ref{line:foil_damper_unstable}
          indicates an unstable simulation.}
 \label{FIG: FOIL DAMPER STAB}
\end{figure}
\begin{figure}
 \centering
 \includegraphics[width=0.45\textwidth]{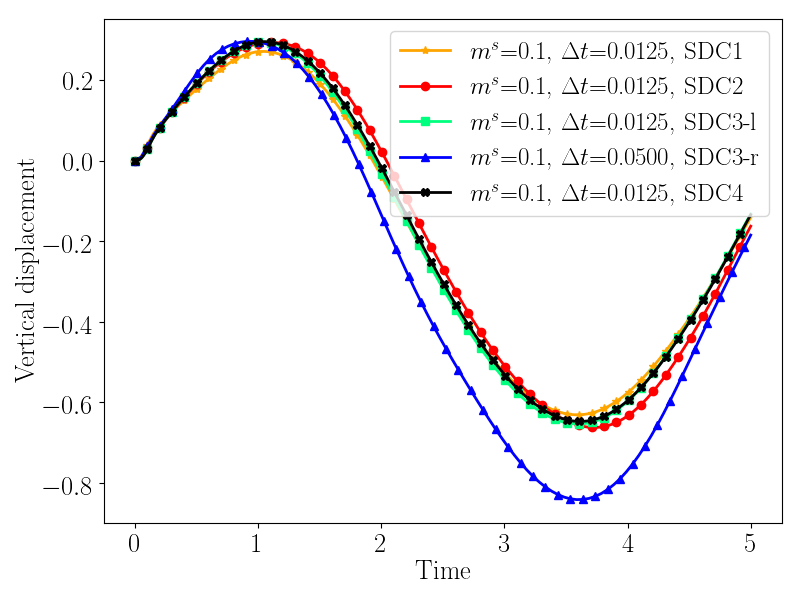} \quad
 \includegraphics[width=0.45\textwidth]{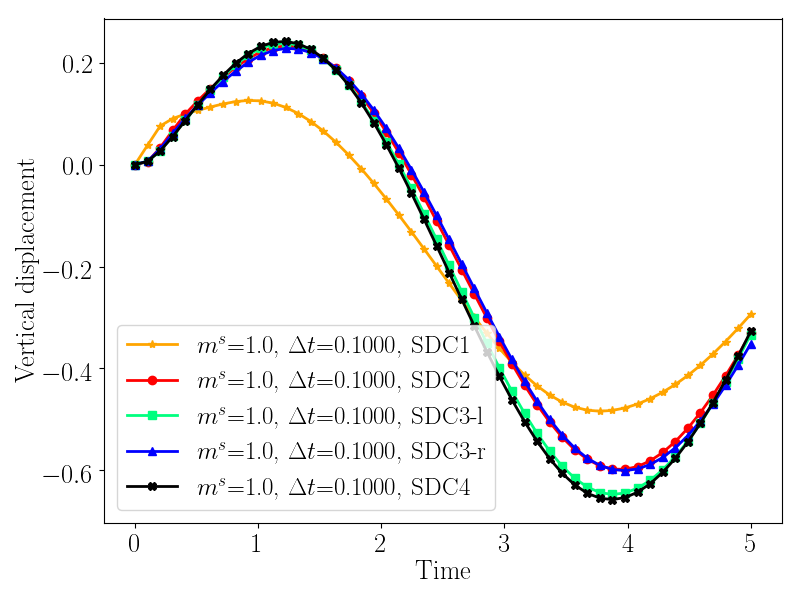}
 \caption{Vertical displacements of the foil for different SDC schemes at the minimum stable structure mass.}
 \label{FIG: FOIL DAMPER DISP}
\end{figure}

Figure~\ref{FIG: FOIL DAMPER STAB} also shows that all schemes are stable once
the time step is sufficiently small, at least for this range of mass ratios
considered.
The SDC3-r scheme is the most robust scheme, which is the same as the result in \cref{sec: modified cavity}.
Moreover, by comparing with the IMEX based partition solvers in \cite{Huang2019}, which is reproduced in \cref{FIG:IMEX-FOIL-DAMPER-STAB},
we conclude the SDC-based partition schemes are more robust for all temporal orders.
Finally, the efficiency of these two arbitrarily high-order partitioned solvers is also studied, in terms of the number of implicit solvings,
which are listed in \cref{tab:efficiency_sdc_imex}.  The IMEX schemes are more efficient than SDC schemes in terms of implicit solves; however, potential improvement for SDC
would be parallel-in-time evaluation, using an algorithm such as PFASST \cite{Emmett2012}.

\begin{figure}
 \centering
 \input{foil_damper_stability_imex.tikz}
 \caption{Behavior of the predictor-based partitioned schemes for a range
          of mass ratios and time steps for IMEX1-IMEX4 (\emph{left} to
          \emph{right}) schemes with the weak Gauss-Seidel predictor. \emph{Legend}: \ref{line:foil_damper_stable_imex}
          indicates a stable simulation and \ref{line:foil_damper_unstable_imex}
          indicates an unstable simulation~\cite{Huang2019}.}
 \label{FIG:IMEX-FOIL-DAMPER-STAB}
\end{figure}
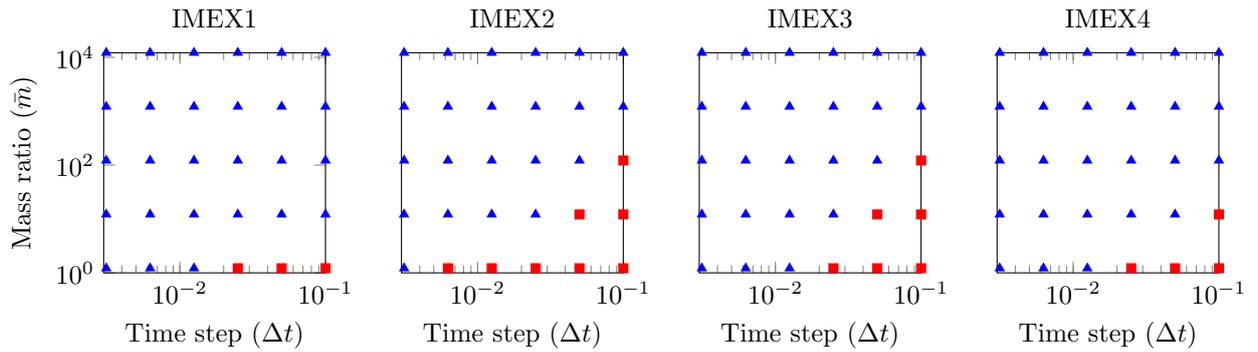

\begin{table}
\begin{tabular}{c|c|cccc|cccc}
\toprule[1.5pt]
Method    & BE & SDC1 &  SDC2 & SDC3-l/r & SDC4 & IMEX1 & IMEX2 & IMEX3 &IMEX4\\
\hline
Implicit fluid solve       & 1 & 1 &  2 & 6 &  8  & 1 & 1 & 3 &  5    \\
Implicit structure solve  & 1 & 1 &  2 & 6 &  8  & 1 & 1 & 3 &  5   \\
\bottomrule[1.5pt]
\end{tabular}
\caption{Number of Implicit solvings during one time step for SDC and IMEX based partitioned solvers}
\label{tab:efficiency_sdc_imex}
\end{table}

\section{Conclusion}
This paper introduces a framework for constructing high-order, stable, partitioned solvers for
general multiphysics problems, whose governing PDEs are first discretized in space only to a set of first-order ODEs.
The ODE system is solved by spectral deferred correction methods, wherein arbitrarily high-order accuracy is attained by performing a series of correction
sweeps using a low-order solver.  When the low-order solver is designed to be partitioned, the corresponding SDC solver is partitioned.
Moreover, thanks to these correction sweeps or iterations, the resultant SDC solvers are more stable, which has been demonstrated in the present work.
Sufficient conditions to construct consistent low-order solvers are derived and based on these conditions, partitioned multiphysics solvers up to fourth order are constructed,  and their properties are analyzed in detail. The stability property of SDC1 is  studied based on couple linear model problem, which can be used to guide the design of more robust partitioned solvers. Further stability analysis for different choices of low-order approximation will be considered in the future. The number of implicit solvings increases quadratically along with the increasing of the order of accuracy, how to improve the efficiency is also worth future investigations.

\section*{Acknowlegement}
Daniel Z. Huang gratefully acknowledges support from the Jet Propulsion Laboratory (JPL) under Contract JPL-RSA No. 1590208.
This work was also supported by the National Aeronautics and Space Administration (NASA) under grant number NNX16AP15A.
Lawrence Livermore National Laboratory is operated by Lawrence Livermore National Security, LLC, for the U.S. Department of Energy,
National Nuclear Security Administration under Contract DE-AC52-07NA27344. LLNL-JRNL-788544.

\appendix

{
\section{Analytical comparison of different schemes for the ODE system in \Cref{SEC: APP ODE} }
\label{sec: proof}
Let $a_{ij}\,, i=1,2\,,j=1,2$ be the entries of the coefficient matrix in the ODE system in \Cref{SEC: APP ODE}, the  monolithic forward Euler scheme, the partitioned SDC1 scheme, and the fixed point iteration scheme are analytically studied in this section.

The monolithic forward Euler scheme updates the state as
\begin{equation}
\ubm_{n+1} = \ubm_{n} + \Delta t \bm{A} \ubm_{n}.
\end{equation}
The update matrix becomes
\begin{equation}
\bm{C} =  \bm{I} + \Delta t \bm{A},
\end{equation}
and its spectrum radius is $\rho(\bm{C}) = \max\{|1 - \Delta t|, |1 - \Delta t \alpha|\}$.
Therefore, the maximum time step is $\Delta t_{\max} = \min\{\frac{2}{\alpha}, 2\}$.

The partitioned SDC1 scheme~(1 iteration) in \cref{ALG: SDC WEAK GS} updates the state as
\begin{equation}
\begin{aligned}
&u_{n+1}^{1} = u_{n}^{1} +  \Delta t\Big(\big(a^{11} u_{n+1}^{1} + a^{12} u_{n}^{2}\big) -  \big(a^{11} u_{n}^{1} + a^{12} u_{n}^{2}\big) \Big) + \Delta t  \big(a^{11} u_{n}^{1} + a^{12} u_{n}^{2}\big) \\
&u_{n+1}^{2} = u_{n}^{2} +  \Delta t\Big(\big(a^{21} u_{n+1}^{1} + a^{22} u_{n+1}^{2}\big) -  \big(a^{21} u_{n}^{1} + a^{22} u_{n}^{2}\big) \Big)  + \Delta t \big(a^{21} u_{n}^{1} + a^{22} u_{n}^{2}\big).
\end{aligned}
\end{equation}
Bringing the parameters in \Cref{SEC: APP ODE}, the update matrix becomes
\begin{equation}
\bm{C} =  \begin{bmatrix}
 1 & \Delta t \\
 -\frac{\alpha \Delta t}{1 + \Delta t(\alpha + 1)} & \frac{1 - \alpha \Delta t^2}{1 + \Delta t(\alpha + 1)}
\end{bmatrix},
\end{equation}
and its eigenvalues satisfy
\begin{equation}
\label{EQ:LAMBDA2}
\lambda^2 - \Big(1 + \frac{1 - \alpha \Delta t^2}{1 + \Delta t (\alpha + 1)} \Big)\lambda + \frac{1}{1 + \Delta t (\alpha + 1)} = 0,
\end{equation}
If \cref{EQ:LAMBDA2} has a pair of  complex-conjugate eigenvalues, we have $\lambda_1 \lambda_2 = \frac{1}{1 + \Delta t (\alpha + 1)} < 1$, therefore it is unconditional stable.
If \cref{EQ:LAMBDA2} has positive eigenvalues, \cref{EQ:LAMBDA2} can be rearranged as
\begin{equation}
\Big(\lambda - 1\Big)\Big(\lambda - \frac{1}{1 + \Delta t (\alpha + 1)} \Big) + \frac{\alpha \Delta t^2}{1 + \Delta t (\alpha + 1)}\lambda = 0,
\end{equation}
therefore, the positive eigenvalue is smaller than $1$. And the only unstable case corresponds to eigenvalues smaller than $-1$.
Therefore, the maximum time step is $\Delta t_{\max} = \frac{\alpha + 1 + \sqrt{(\alpha + 1)^2 + 4\alpha}}{\alpha}$.

The fixed point iteration scheme with the same type of Gauss-Seidel coupling starts with a suitably predicted  $\tilde{u}_{n+1}^{2} = u_{n}^{2}$, and solves each subsystem equation
\begin{equation}
\begin{aligned}
&u_{n+1}^{1} = u_{n}^{1} +  \Delta t\Big(a^{11} u_{n+1}^{1} + a^{12} \tilde{u}_{n+1}^{2}\Big),\\
&u_{n+1}^{2} = u_{n}^{2} +  \Delta t\Big(a^{21} u_{n+1}^{1} + a^{22} u_{n+1}^{2} \Big)  .
\end{aligned}
\end{equation}
Then it updates the predicted state $\tilde{u}_{n+1}^{2} := u_{n+1}^{2}$ until it converges.
The update equation of the fixed point iteration becomes
\begin{equation}
\begin{aligned}
u_{n+1}^{2} = \frac{\Delta t^2 a^{12} a^{21}}{\big(1 - \Delta t a^{11}\big)\big(1 - \Delta t a^{22}\big)}\tilde{u}_{n+1}^{2} +  \frac{1}{1 - \Delta t a^{22}}u_{n}^{2} + \frac{\Delta t a^{21}}{\big(1 - \Delta t a^{11}\big)\big(1 - \Delta t a^{22}\big)} {u}_{n}^{1}.
\end{aligned}
\end{equation}
Bringing the parameters in \Cref{SEC: APP ODE}, the fixed point iteration converges with  $\Delta t_{\max} = \frac{\alpha + 1 + \sqrt{(\alpha + 1)^2 + 4\alpha}}{2\alpha}$, and the linear convergence rate is
\begin{equation}
\mathcal{O}(\frac{\Delta t^2 \alpha}{1 + \Delta t(\alpha + 1)}).
\end{equation}

For this specific ODE model problem, our partitioned SDC1 scheme has larger stable time step. As for high order partitioned SDC schemes, even the iteration number increases quadratically, for 1-st to 4-th order SDC schemes, the iteration numbers might be fewer than that of the fixed point iteration.

}

\newpage
\bibliography{refs}
\bibliographystyle{unsrt}

\end{document}

%% file: preamble.tex
\usepackage{bm}
\usepackage{amsmath}
\usepackage{amsfonts}
\usepackage{amsthm}
\usepackage{amssymb}
\usepackage{graphicx}
\usepackage{paralist}
\usepackage{algorithm}
\usepackage{pstool}
\usepackage{wrapfig}
\usepackage{subcaption}
\usepackage{algpseudocode}
\usepackage{mathtools}
\usepackage{fullpage}
\usepackage{etoolbox}

\usepackage{tikz}
\usepackage{pgfplots}
\usepackage{pgfplotstable, filecontents}
\pgfplotsset{compat=1.9}
\usetikzlibrary{pgfplots.groupplots}
\usepgfplotslibrary{fillbetween}
\usetikzlibrary{calc,fit,matrix,arrows,automata,positioning,shapes}
\usepackage{pgfplotstable, booktabs}
\usetikzlibrary{patterns,decorations.pathmorphing,decorations.markings,decorations.pathmorphing}
\pgfplotsset{select coords between index/.style 2 args={
    x filter/.code={
        \ifnum\coordindex<#1\fi
        \ifnum\coordindex>#2\fi
    }
}}

\pgfplotsset{compat=newest}

\usetikzlibrary{calc}

\newcommand{\logLogSlopeTriangle}[5]
{

    \pgfplotsextra
    {
        \pgfkeysgetvalue{/pgfplots/xmin}{\xmin}
        \pgfkeysgetvalue{/pgfplots/xmax}{\xmax}
        \pgfkeysgetvalue{/pgfplots/ymin}{\ymin}
        \pgfkeysgetvalue{/pgfplots/ymax}{\ymax}

        \pgfmathsetmacro{\xArel}{#1}
        \pgfmathsetmacro{\yArel}{#3}
        \pgfmathsetmacro{\xBrel}{#1-#2}
        \pgfmathsetmacro{\yBrel}{\yArel}
        \pgfmathsetmacro{\xCrel}{\xArel}

        \pgfmathsetmacro{\lnxB}{\xmin*(1-(#1-#2))+\xmax*(#1-#2)} 
        \pgfmathsetmacro{\lnxA}{\xmin*(1-#1)+\xmax*#1} 
        \pgfmathsetmacro{\lnyA}{\ymin*(1-#3)+\ymax*#3} 
        \pgfmathsetmacro{\lnyC}{\lnyA+#4*(\lnxA-\lnxB)}
        \pgfmathsetmacro{\yCrel}{\lnyC-\ymin)/(\ymax-\ymin)} 

        \coordinate (A) at (rel axis cs:\xArel,\yArel);
        \coordinate (B) at (rel axis cs:\xBrel,\yBrel);
        \coordinate (C) at (rel axis cs:\xCrel,\yCrel);

        \draw[#5]   (A)-- node[pos=0.5,anchor=north,scale=0.8] {1}
                    (B)-- 
                    (C)-- node[pos=0.5,anchor=west,scale=0.8] {#4}
                    cycle;
    }
}

\newcommand{\dmu}[1][]{\ifthenelse{\isempty{#1}}{\partial_\mubold}{\pder{#1}{\mubold}}}
\newcommand{\mass}[1][]{\ifthenelse{\isempty{#1}}{\Mbm}{\Mbm^{#1}}}
\newcommand{\stvc}[1][]{\ifthenelse{\isempty{#1}}{\ubm}{\ubm^{#1}}}
\newcommand{\stvcdot}[1][]{\ifthenelse{\isempty{#1}}{\dot\ubm}{\dot\ubm^{#1}}}
\newcommand{\stvcbar}[1][]{\ifthenelse{\isempty{#1}}{\bar\ubm}{\bar\ubm^{#1}}}
\newcommand{\res}[1][]{\ifthenelse{\isempty{#1}}{\rbm}{\rbm^{#1}}}
\newcommand{\resimpl}[1][]{\ifthenelse{\isempty{#1}}{\gbm}{\gbm^{#1}}}
\newcommand{\resexpl}[1][]{\ifthenelse{\isempty{#1}}{\fbm}{\fbm^{#1}}}
\newcommand{\cpl}[1][]{\ifthenelse{\isempty{#1}}{\cbm}{\cbm^{#1}}}
\newcommand{\cplprd}[1][]{\ifthenelse{\isempty{#1}}{\tilde\cbm}{\tilde\cbm^{#1}}}

\newcommand{\pstp}[2][]{\ifthenelse{\isempty{#2}}{\ubm_{#1}}{\ubm_{#2}^{#1}}}
\newcommand{\pstgki}[3][]{\ifthenelse{\isempty{#3}}{\kbm_{#1,#2}}{\kbm_{#2,#3}^{#1}}}
\newcommand{\pstgke}[3][]{\ifthenelse{\isempty{#3}}{\hat\kbm_{#1,#2}}{\hat\kbm_{#2,#3}^{#1}}}
\newcommand{\pstgu}[3][]{\ifthenelse{\isempty{#3}}{\ubm_{#1,#2}}{\ubm_{#2,#3}^{#1}}}
\newcommand{\cplstp}[2][]{\ifthenelse{\isempty{#2}}{\cbm_{#1}}{\cbm_{#2}^{#1}}}
\newcommand{\cplprdstp}[2][]{\ifthenelse{\isempty{#2}}{\tilde\cbm_{#1}}{\tilde\cbm_{#2}^{#1}}}

\newcommand{\dt}[1]{\Delta t_{#1}}

\newcommand{\qstg}[3][]{\ifthenelse{\isempty{#3}}{\qbm_{#1,#2}}{\qbm_{#2,#3}^{#1}}}
\newcommand{\prstp}[2][]{\ifthenelse{\isempty{#2}}{\pbm_{#1}}{\pbm_{#2}^{#1}}}

\newcommand{\pder}[2]{\ensuremath{\frac{\partial #1}{\partial #2}}} 

\newcommand{\Fcal}{\ensuremath{\mathcal{F}}}
\newcommand{\Gcal}{\ensuremath{\mathcal{G}}}

\newcommand{\Rbb}{\ensuremath{\mathbb{R} }}

\newcommand\Mbm{{\ensuremath{\bm{M}}}}

\newcommand\cbm{{\ensuremath{\bm{c}}}}

\newcommand\fbm{{\ensuremath{\bm{f}}}}
\newcommand\gbm{{\ensuremath{\bm{g}}}}

\newcommand\kbm{{\ensuremath{\bm{k}}}}

\newcommand\pbm{{\ensuremath{\bm{p}}}}
\newcommand\qbm{{\ensuremath{\bm{q}}}}
\newcommand\rbm{{\ensuremath{\bm{r}}}}

\newcommand\ubm{{\ensuremath{\bm{u}}}}

\newcommand\xbm{{\ensuremath{\bm{x}}}}

\newcommand\mubold{{\ensuremath{\boldsymbol{\mu}}}}



%% file: ode_3field_GS.tikz
\begin{tikzpicture}
\begin{loglogaxis}[
    width=0.45\textwidth,
    height=0.45\textwidth,
    xlabel={Time step ($\Delta t$)},
    ylabel={~~}]

\addplot [orange, solid, thick, mark=*, mark size=1, mark options={solid}]  table[x index=0, y index=1] {ode_3fields_GS.dat};\label{line:ode_3fields:sdc1}
\addplot [red, solid, thick, mark=square*, mark size=1, mark options={solid}]  table[x index=0, y index=2] {ode_3fields_GS.dat};\label{line:ode_3fields:sdc2}
\addplot [green, solid, thick, mark=otimes*, mark size=1, mark options={solid}]  table[x index=0, y index=4] {ode_3fields_GS.dat};\label{line:ode_3fields:sdc3-r}
\addplot [blue, solid, thick, mark=triangle*, mark size=1, mark options={solid}]  table[x index=0, y index=3] {ode_3fields_GS.dat};\label{line:ode_3fields:sdc3-l}
\addplot [black, solid, thick, mark=diamond*, mark size=1, mark options={solid}]  table[x index=0, y index=5] {ode_3fields_GS.dat};\label{line:ode_3fields:sdc4}

\logLogSlopeTriangle{0.2}{0.05}{0.45}{1}{orange};
\logLogSlopeTriangle{0.15}{0.05}{0.30}{2}{red};
\logLogSlopeTriangle{0.15}{0.05}{0.24}{3}{green};
\logLogSlopeTriangle{0.15}{0.05}{0.18}{3}{blue};
\logLogSlopeTriangle{0.2}{0.05}{0.12}{4}{black};

\end{loglogaxis}

\end{tikzpicture}

%% file: predator_prey_GS.tikz
\begin{tikzpicture}
\begin{loglogaxis}[
    width=0.45\textwidth,
    height=0.45\textwidth,
    xlabel={Time step ($\Delta t$)},
    ylabel={~~}]

\addplot [orange, solid, thick, mark=*, mark size=1, mark options={solid}]  table[x index=0, y index=1] {predator_prey_GS.dat};\label{line:predator_prey:sdc1}
\addplot [red, solid, thick, mark=square*, mark size=1, mark options={solid}]  table[x index=0, y index=2] {predator_prey_GS.dat};\label{line:predator_prey:sdc2}
\addplot [green, solid, thick, mark=otimes*, mark size=1, mark options={solid}]  table[x index=0, y index=4] {predator_prey_GS.dat};\label{line:predator_prey:sdc3-r}
\addplot [blue, solid, thick, mark=triangle*, mark size=1, mark options={solid}]  table[x index=0, y index=3] {predator_prey_GS.dat};\label{line:predator_prey:sdc3-l}
\addplot [black, solid, thick, mark=diamond*, mark size=1, mark options={solid}]  table[x index=5, y index=6] {predator_prey_GS.dat};\label{line:predator_prey:sdc4}

\logLogSlopeTriangle{0.4}{0.1}{0.8}{1}{orange};
\logLogSlopeTriangle{0.4}{0.1}{0.6}{2}{red};
\logLogSlopeTriangle{0.4}{0.1}{0.5}{3}{green};
\logLogSlopeTriangle{0.4}{0.1}{0.3}{3}{blue};
\logLogSlopeTriangle{0.4}{0.1}{0.1}{4}{black};

\end{loglogaxis}

\end{tikzpicture}

%% file: foil-with-damper.tikz
\begin{tikzpicture}
\begin{axis}[
axis equal,
axis lines=none,
width=8cm,
ymax=0.6,
xmax=0.85,
xmin=-0.85,
ymin=-0.6]
\addplot [white, forget plot]
coordinates {
( -0.85000000,  -0.60000000)
(  0.85000000,  -0.60000000)
(  0.85000000,   0.60000000)
( -0.85000000,   0.60000000)
( -0.85000000,  -0.60000000)};

\draw [smooth, ultra thin] plot coordinates {(axis cs:-1.000000, -0.580000) (axis cs:1.000000, -0.580000)
};

\draw [smooth, ultra thin, opacity=0.5, fill=black!30!white] plot coordinates {(axis cs:-0.385791, 0.103372) (axis cs:-0.371603, 0.117293) (axis cs:-0.360141, 0.121044) (axis cs:-0.349136, 0.123087) (axis cs:-0.338373, 0.124230) (axis cs:-0.327767, 0.124784) (axis cs:-0.317276, 0.124909) (axis cs:-0.306874, 0.124701) (axis cs:-0.296545, 0.124224) (axis cs:-0.286276, 0.123522) (axis cs:-0.276059, 0.122625) (axis cs:-0.265887, 0.121559) (axis cs:-0.255756, 0.120343) (axis cs:-0.245660, 0.118992) (axis cs:-0.235597, 0.117520) (axis cs:-0.225564, 0.115935) (axis cs:-0.215559, 0.114249) (axis cs:-0.205579, 0.112467) (axis cs:-0.195623, 0.110598) (axis cs:-0.185688, 0.108646) (axis cs:-0.175775, 0.106617) (axis cs:-0.165880, 0.104516) (axis cs:-0.156004, 0.102346) (axis cs:-0.146146, 0.100111) (axis cs:-0.136304, 0.097816) (axis cs:-0.126477, 0.095462) (axis cs:-0.116665, 0.093053) (axis cs:-0.106868, 0.090591) (axis cs:-0.097083, 0.088079) (axis cs:-0.087312, 0.085518) (axis cs:-0.077553, 0.082912) (axis cs:-0.067806, 0.080262) (axis cs:-0.058070, 0.077570) (axis cs:-0.048345, 0.074838) (axis cs:-0.038630, 0.072066) (axis cs:-0.028925, 0.069258) (axis cs:-0.019230, 0.066413) (axis cs:-0.009544, 0.063535) (axis cs:0.000133, 0.060623) (axis cs:0.009802, 0.057679) (axis cs:0.019462, 0.054704) (axis cs:0.029114, 0.051700) (axis cs:0.038759, 0.048667) (axis cs:0.048396, 0.045606) (axis cs:0.058026, 0.042519) (axis cs:0.067649, 0.039405) (axis cs:0.077266, 0.036267) (axis cs:0.086876, 0.033105) (axis cs:0.096480, 0.029920) (axis cs:0.106077, 0.026711) (axis cs:0.115669, 0.023481) (axis cs:0.125255, 0.020230) (axis cs:0.134836, 0.016958) (axis cs:0.144411, 0.013666) (axis cs:0.153981, 0.010355) (axis cs:0.163546, 0.007025) (axis cs:0.173106, 0.003676) (axis cs:0.182662, 0.000310) (axis cs:0.192212, -0.003074) (axis cs:0.201758, -0.006475) (axis cs:0.211300, -0.009893) (axis cs:0.220837, -0.013327) (axis cs:0.230370, -0.016776) (axis cs:0.239899, -0.020241) (axis cs:0.249424, -0.023722) (axis cs:0.258945, -0.027217) (axis cs:0.268462, -0.030727) (axis cs:0.277975, -0.034251) (axis cs:0.287484, -0.037789) (axis cs:0.296989, -0.041342) (axis cs:0.306491, -0.044908) (axis cs:0.315989, -0.048488) (axis cs:0.325484, -0.052082) (axis cs:0.334974, -0.055688) (axis cs:0.344462, -0.059309) (axis cs:0.353946, -0.062942) (axis cs:0.363426, -0.066589) (axis cs:0.372903, -0.070248) (axis cs:0.382376, -0.073921) (axis cs:0.391845, -0.077607) (axis cs:0.401312, -0.081306) (axis cs:0.410774, -0.085018) (axis cs:0.420234, -0.088743) (axis cs:0.429689, -0.092482) (axis cs:0.439141, -0.096234) (axis cs:0.448590, -0.099999) (axis cs:0.458034, -0.103778) (axis cs:0.467475, -0.107571) (axis cs:0.476913, -0.111377) (axis cs:0.486346, -0.115198) (axis cs:0.495776, -0.119033) (axis cs:0.505202, -0.122882) (axis cs:0.514624, -0.126747) (axis cs:0.524042, -0.130626) (axis cs:0.533456, -0.134520) (axis cs:0.542865, -0.138430) (axis cs:0.552271, -0.142355) (axis cs:0.561672, -0.146297) (axis cs:0.571069, -0.150255) (axis cs:0.579809, -0.156664) (axis cs:0.569688, -0.155410) (axis cs:0.559571, -0.154139) (axis cs:0.549458, -0.152852) (axis cs:0.539350, -0.151549) (axis cs:0.529246, -0.150230) (axis cs:0.519146, -0.148896) (axis cs:0.509051, -0.147546) (axis cs:0.498959, -0.146182) (axis cs:0.488871, -0.144802) (axis cs:0.478787, -0.143409) (axis cs:0.468707, -0.142001) (axis cs:0.458631, -0.140579) (axis cs:0.448558, -0.139143) (axis cs:0.438489, -0.137693) (axis cs:0.428424, -0.136230) (axis cs:0.418363, -0.134753) (axis cs:0.408304, -0.133263) (axis cs:0.398250, -0.131760) (axis cs:0.388199, -0.130243) (axis cs:0.378152, -0.128713) (axis cs:0.368108, -0.127170) (axis cs:0.358067, -0.125615) (axis cs:0.348030, -0.124046) (axis cs:0.337997, -0.122463) (axis cs:0.327967, -0.120868) (axis cs:0.317941, -0.119260) (axis cs:0.307918, -0.117638) (axis cs:0.297899, -0.116003) (axis cs:0.287883, -0.114354) (axis cs:0.277871, -0.112691) (axis cs:0.267863, -0.111015) (axis cs:0.257859, -0.109325) (axis cs:0.247858, -0.107621) (axis cs:0.237861, -0.105902) (axis cs:0.227868, -0.104169) (axis cs:0.217879, -0.102420) (axis cs:0.207895, -0.100657) (axis cs:0.197914, -0.098878) (axis cs:0.187938, -0.097083) (axis cs:0.177966, -0.095271) (axis cs:0.167998, -0.093444) (axis cs:0.158035, -0.091599) (axis cs:0.148077, -0.089737) (axis cs:0.138123, -0.087857) (axis cs:0.128174, -0.085958) (axis cs:0.118231, -0.084041) (axis cs:0.108292, -0.082104) (axis cs:0.098359, -0.080148) (axis cs:0.088432, -0.078170) (axis cs:0.078510, -0.076172) (axis cs:0.068594, -0.074151) (axis cs:0.058684, -0.072108) (axis cs:0.048781, -0.070041) (axis cs:0.038883, -0.067951) (axis cs:0.028993, -0.065835) (axis cs:0.019109, -0.063694) (axis cs:0.009233, -0.061526) (axis cs:-0.000636, -0.059330) (axis cs:-0.010497, -0.057106) (axis cs:-0.020351, -0.054852) (axis cs:-0.030196, -0.052567) (axis cs:-0.040033, -0.050251) (axis cs:-0.049860, -0.047901) (axis cs:-0.059679, -0.045516) (axis cs:-0.069488, -0.043096) (axis cs:-0.079286, -0.040639) (axis cs:-0.089075, -0.038143) (axis cs:-0.098853, -0.035606) (axis cs:-0.108619, -0.033028) (axis cs:-0.118374, -0.030405) (axis cs:-0.128116, -0.027737) (axis cs:-0.137845, -0.025020) (axis cs:-0.147561, -0.022253) (axis cs:-0.157263, -0.019434) (axis cs:-0.166950, -0.016559) (axis cs:-0.176622, -0.013626) (axis cs:-0.186277, -0.010632) (axis cs:-0.195914, -0.007573) (axis cs:-0.205534, -0.004446) (axis cs:-0.215134, -0.001246) (axis cs:-0.224713, 0.002031) (axis cs:-0.234270, 0.005390) (axis cs:-0.243804, 0.008837) (axis cs:-0.253312, 0.012379) (axis cs:-0.262793, 0.016024) (axis cs:-0.272244, 0.019780) (axis cs:-0.281662, 0.023658) (axis cs:-0.291044, 0.027670) (axis cs:-0.300386, 0.031833) (axis cs:-0.309683, 0.036165) (axis cs:-0.318928, 0.040691) (axis cs:-0.328112, 0.045442) (axis cs:-0.337224, 0.050464) (axis cs:-0.346246, 0.055818) (axis cs:-0.355154, 0.061600) (axis cs:-0.363904, 0.067971) (axis cs:-0.372413, 0.075244) (axis cs:-0.380465, 0.084223) (axis cs:-0.385791, 0.103372)
};

\draw [smooth, ultra thin] plot coordinates {(axis cs:-0.050000, -0.325000) (axis cs:0.050000, -0.325000)
};

\draw [smooth, ultra thin] plot coordinates {(axis cs:0.000000, -0.580000) (axis cs:0.000000, -0.325000)
};

\draw [smooth, ultra thin] plot coordinates {(axis cs:-0.060000, -0.250000) (axis cs:-0.060000, -0.400000)
};

\draw [smooth, ultra thin] plot coordinates {(axis cs:-0.060000, -0.250000) (axis cs:0.060000, -0.250000)
};

\draw [smooth, ultra thin] plot coordinates {(axis cs:0.060000, -0.250000) (axis cs:0.060000, -0.400000)
};

\draw [smooth, ultra thin] plot coordinates {(axis cs:0.000000, -0.250000) (axis cs:0.000000, 0.000000)
};

\node[]    at    (axis cs:0.1, -0.025) {$m^s$};
\node[]    at    (axis cs:0.135, -0.325) {$c^s$};
\draw [smooth, solid, gray, ultra thin] plot coordinates {(axis cs:0.000600, 0.000000) (axis cs:-0.799400, 0.000000)
};

\draw [smooth, solid, gray, ultra thin] plot coordinates {(axis cs:0.000580, -0.000155) (axis cs:-0.772161, 0.206900)
};

\node[]    at    (axis cs:-0.6, 0.075) {$\theta(t)$};
\node[]    at    (axis cs:-0.58, -0.325) {$d^s$};
\draw [smooth, solid, gray, <->, ultra thin] plot coordinates {(axis cs:-0.650000, -0.580000) (axis cs:-0.650000, 0.000000)
};

\end{axis}
\end{tikzpicture}

%% file: foil_damper_stability.tikz
\begin{tikzpicture}
\begin{groupplot}[
    group style={
        group name=conv plots,
        group size=5 by 1,
        horizontal sep=1.0cm
    },
    xmode=log, ymode=log,
    xmin=0.003, xmax=0.1,
    ymin=1, ymax=1000.0/0.08221,
    width=3.7cm,
    height=3.7cm,
]

\nextgroupplot[ylabel={Mass ratio $(\bar{m})$}, xlabel={Time step ($\Delta t$)},title=SDC1]
 \addplot [blue, solid, thick, mark=triangle*, mark size=1.5, only marks, mark options={solid}]  table[x index=1, y expr=\thisrowno{0}/0.08221] {nacamsh1ref0p3_Re1000_SDC1.stable.dat};
 \label{line:foil_damper_stable}
 \addplot [red, solid, thick, mark=square*, mark size=1.5, only marks, mark options={solid}]  table[x index=1, y expr=\thisrowno{0}/0.08221] {nacamsh1ref0p3_Re1000_SDC1.unstable.dat};
\label{line:foil_damper_unstable}

\nextgroupplot[xlabel={Time step ($\Delta t$)}, ytick=\empty,title=SDC2]
 \addplot [blue, solid, thick, mark=triangle*, mark size=1.5, only marks, mark options={solid}]  table[x index=1, y expr=\thisrowno{0}/0.08221] {nacamsh1ref0p3_Re1000_SDC2.stable.dat};
 \addplot [red, solid, thick, mark=square*, mark size=1.5, only marks, mark options={solid}]  table[x index=1, y expr=\thisrowno{0}/0.08221] {nacamsh1ref0p3_Re1000_SDC2.unstable.dat};

\nextgroupplot[xlabel={Time step ($\Delta t$)}, ytick=\empty,title=SDC3-l]
 \addplot [blue, solid, thick, mark=triangle*, mark size=1.5, only marks, mark options={solid}]  table[x index=1, y expr=\thisrowno{0}/0.08221] {nacamsh1ref0p3_Re1000_SDC3l.stable.dat};
 \addplot [red, solid, thick, mark=square*, mark size=1.5, only marks, mark options={solid}]  table[x index=1, y expr=\thisrowno{0}/0.08221] {nacamsh1ref0p3_Re1000_SDC3l.unstable.dat};

\nextgroupplot[xlabel={Time step ($\Delta t$)}, ytick=\empty,title=SDC3-r]
 \addplot [blue, solid, thick, mark=triangle*, mark size=1.5, only marks, mark options={solid}]  table[x index=1, y expr=\thisrowno{0}/0.08221] {nacamsh1ref0p3_Re1000_SDC3r.stable.dat};
 \addplot [red, solid, thick, mark=square*, mark size=1.5, only marks, mark options={solid}]  table[x index=1, y expr=\thisrowno{0}/0.08221] {nacamsh1ref0p3_Re1000_SDC3r.unstable.dat};

\nextgroupplot[xlabel={Time step ($\Delta t$)}, ytick=\empty,title=SDC4]
 \addplot [blue, solid, thick, mark=triangle*, mark size=1.5, only marks, mark options={solid}]  table[x index=1, y expr=\thisrowno{0}/0.08221] {nacamsh1ref0p3_Re1000_SDC4.stable.dat};
 \addplot [red, solid, thick, mark=square*, mark size=1.5, only marks, mark options={solid}]  table[x index=1, y expr=\thisrowno{0}/0.08221] {nacamsh1ref0p3_Re1000_SDC4.unstable.dat};

\end{groupplot}
\end{tikzpicture}

%% file: foil_damper_stability_imex.tikz
\begin{tikzpicture}
\begin{groupplot}[
    group style={
        group name=conv plots,
        group size=4 by 4,
        horizontal sep=1.0cm
    },
    xmode=log, ymode=log,
    xmin=0.003, xmax=0.1,
    ymin=1, ymax=1000.0/0.08221,
    width=4.5cm,
    height=4.5cm,
]

\nextgroupplot[ylabel={Mass ratio $(\bar{m})$}, xlabel={Time step ($\Delta t$)},title=IMEX1]
 \addplot [blue, solid, thick, mark=triangle*, mark size=1.5, only marks, mark options={solid}]  table[x index=1, y expr=\thisrowno{0}/0.08221] {nacamsh1ref0p3_Re1000_IMEX1.stable.dat};
 \label{line:foil_damper_stable_imex}
 \addplot [red, solid, thick, mark=square*, mark size=1.5, only marks, mark options={solid}]  table[x index=1, y expr=\thisrowno{0}/0.08221] {nacamsh1ref0p3_Re1000_IMEX1.unstable.dat};
\label{line:foil_damper_unstable_imex}

\nextgroupplot[xlabel={Time step ($\Delta t$)}, ytick=\empty,title=IMEX2]
 \addplot [blue, solid, thick, mark=triangle*, mark size=1.5, only marks, mark options={solid}]  table[x index=1, y expr=\thisrowno{0}/0.08221] {nacamsh1ref0p3_Re1000_IMEX2.stable.dat};
 \addplot [red, solid, thick, mark=square*, mark size=1.5, only marks, mark options={solid}]  table[x index=1, y expr=\thisrowno{0}/0.08221] {nacamsh1ref0p3_Re1000_IMEX2.unstable.dat};

\nextgroupplot[xlabel={Time step ($\Delta t$)}, ytick=\empty,title=IMEX3]
 \addplot [blue, solid, thick, mark=triangle*, mark size=1.5, only marks, mark options={solid}]  table[x index=1, y expr=\thisrowno{0}/0.08221] {nacamsh1ref0p3_Re1000_IMEX3.stable.dat};
 \addplot [red, solid, thick, mark=square*, mark size=1.5, only marks, mark options={solid}]  table[x index=1, y expr=\thisrowno{0}/0.08221] {nacamsh1ref0p3_Re1000_IMEX3.unstable.dat};

\nextgroupplot[xlabel={Time step ($\Delta t$)}, ytick=\empty,title=IMEX4]
 \addplot [blue, solid, thick, mark=triangle*, mark size=1.5, only marks, mark options={solid}]  table[x index=1, y expr=\thisrowno{0}/0.08221] {nacamsh1ref0p3_Re1000_IMEX4.stable.dat};
 \addplot [red, solid, thick, mark=square*, mark size=1.5, only marks, mark options={solid}]  table[x index=1, y expr=\thisrowno{0}/0.08221] {nacamsh1ref0p3_Re1000_IMEX4.unstable.dat};

\end{groupplot}
\end{tikzpicture}

%% file: sdcmultiphys.bbl
\begin{thebibliography}{10}

\bibitem{Huang2019}
D.Z. Huang, P.-O. Persson, and M.J. Zahr.
\newblock High-order, linearly stable, partitioned solvers for general
  multiphysics problems based on implicit-explicit {R}unge-{K}utta schemes.
\newblock {\em Computer Methods in Applied Mechanics and Engineering},
  346:674--706, April 2019.

\bibitem{Huang2019b}
Daniel~Z. Huang, Matthew~J. Zahr, and Per-Olof Persson.
\newblock A high-order partitioned solver for general multiphysics problems and
  its applications in optimization.
\newblock In {\em {AIAA} Scitech 2019 Forum}. American Institute of Aeronautics
  and Astronautics, January 2019.

\bibitem{kamakoti2004fluid}
Ramji Kamakoti and Wei Shyy.
\newblock Fluid--structure interaction for aeroelastic applications.
\newblock {\em Progress in Aerospace Sciences}, 40(8):535--558, 2004.

\bibitem{chen2007numerical}
Xiangying Chen, Ge-Cheng Zha, and Ming-Ta Yang.
\newblock Numerical simulation of 3-d wing flutter with fully coupled
  fluid--structural interaction.
\newblock {\em Computers \& fluids}, 36(5):856--867, 2007.

\bibitem{huang2018simulation}
Zhengyu Huang, Philip Avery, Charbel Farhat, Jason Rabinovitch, Armen
  Derkevorkian, and Lee~D Peterson.
\newblock Simulation of parachute inflation dynamics using an eulerian
  computational framework for fluid-structure interfaces evolving in high-speed
  turbulent flows.
\newblock In {\em 2018 AIAA Aerospace Sciences Meeting}, page 1540, 2018.

\bibitem{bazilevs2006isogeometric}
Yuri Bazilevs, Victor~M Calo, Yongjie Zhang, and Thomas~JR Hughes.
\newblock Isogeometric fluid--structure interaction analysis with applications
  to arterial blood flow.
\newblock {\em Computational Mechanics}, 38(4-5):310--322, 2006.

\bibitem{hron2007fluid}
Jaroslav Hron and Martin M{\'a}dl{\'i}k.
\newblock Fluid-structure interaction with applications in biomechanics.
\newblock {\em Nonlinear analysis: real world applications}, 8(5):1431--1458,
  2007.

\bibitem{chabannes2013high}
Vincent Chabannes, Gon{\c{c}}alo Pena, and Christophe Prud'homme.
\newblock High-order fluid--structure interaction in 2d and 3d application to
  blood flow in arteries.
\newblock {\em Journal of Computational and Applied Mathematics}, 246:1--9,
  2013.

\bibitem{moin2006large}
Parviz Moin and Sourabh~V Apte.
\newblock Large-eddy simulation of realistic gas turbine combustors.
\newblock {\em AIAA journal}, 44(4):698--708, 2006.

\bibitem{chen2011multiphysics}
Yen-Sen Chen, TH~Chou, BR~Gu, JS~Wu, Bill Wu, YY~Lian, and Luke Yang.
\newblock Multiphysics simulations of rocket engine combustion.
\newblock {\em Computers \& Fluids}, 45(1):29--36, 2011.

\bibitem{toth2000b}
G{\'a}bor T{\'o}th.
\newblock The $\nabla\cdot b = 0$ constraint in shock-capturing
  magnetohydrodynamics codes.
\newblock {\em Journal of Computational Physics}, 161(2):605--652, 2000.

\bibitem{chacon2002implicit}
Luis Chac{\'o}n, Dana~A Knoll, and JM~Finn.
\newblock An implicit, nonlinear reduced resistive mhd solver.
\newblock {\em Journal of Computational Physics}, 178(1):15--36, 2002.

\bibitem{cyr2013new}
Eric~C Cyr, John~N Shadid, Raymond~S Tuminaro, Roger~P Pawlowski, and Luis
  Chac{\'o}n.
\newblock A new approximate block factorization preconditioner for
  two-dimensional incompressible (reduced) resistive mhd.
\newblock {\em SIAM Journal on Scientific Computing}, 35(3):B701--B730, 2013.

\bibitem{hubner2004monolithic}
Bj{\"o}rn H{\"u}bner, Elmar Walhorn, and Dieter Dinkler.
\newblock A monolithic approach to fluid--structure interaction using
  space--time finite elements.
\newblock {\em Computer methods in applied mechanics and engineering},
  193(23):2087--2104, 2004.

\bibitem{michler2004monolithic}
C~Michler, SJ~Hulshoff, EH~Van~Brummelen, and Ren{\'e} De~Borst.
\newblock A monolithic approach to fluid--structure interaction.
\newblock {\em Computers \& Fluids}, 33(5):839--848, 2004.

\bibitem{kuttler2008fixed}
Ulrich K{\"u}ttler and Wolfgang~A Wall.
\newblock Fixed-point fluid--structure interaction solvers with dynamic
  relaxation.
\newblock {\em Computational mechanics}, 43(1):61--72, 2008.

\bibitem{farhat2000two}
C.~Farhat and M.~Lesoinne.
\newblock Two efficient staggered algorithms for the serial and parallel
  solution of three-dimensional nonlinear transient aeroelastic problems.
\newblock {\em Computer Methods in Applied Mechanics and Engineering},
  182(3):499--515, 2000.

\bibitem{piperno2001partitioned}
S.~Piperno and C.~Farhat.
\newblock Partitioned procedures for the transient solution of coupled
  aeroelastic problems--{P}art {II}: energy transfer analysis and
  three-dimensional applications.
\newblock {\em Computer Methods in Applied Mechanics and Engineering},
  190(24):3147--3170, 2001.

\bibitem{badia2008fluid}
Santiago Badia, Fabio Nobile, and Christian Vergara.
\newblock Fluid--structure partitioned procedures based on {R}obin transmission
  conditions.
\newblock {\em Journal of Computational Physics}, 227(14):7027--7051, 2008.

\bibitem{causin2005added}
Paola Causin, Jean-Fr{\'e}d{\'e}ric Gerbeau, and Fabio Nobile.
\newblock Added-mass effect in the design of partitioned algorithms for
  fluid--structure problems.
\newblock {\em Computer Methods in Applied Mechanics and Engineering},
  194(42-44):4506--4527, 2005.

\bibitem{zhong1996additive}
Xiaolin Zhong.
\newblock Additive semi-implicit {R}unge--{K}utta methods for computing
  high-speed nonequilibrium reactive flows.
\newblock {\em Journal of Computational Physics}, 128(1):19--31, 1996.

\bibitem{ascher1997implicit}
Uri~M Ascher, Steven~J Ruuth, and Raymond~J Spiteri.
\newblock Implicit-explicit {R}unge-{K}utta methods for time-dependent partial
  differential equations.
\newblock {\em Applied Numerical Mathematics}, 25(2-3):151--167, 1997.

\bibitem{van2007higher}
AH~Van~Zuijlen, Aukje de~Boer, and Hester Bijl.
\newblock Higher-order time integration through smooth mesh deformation for
  3{D} fluid--structure interaction simulations.
\newblock {\em Journal of Computational Physics}, 224(1):414--430, 2007.

\bibitem{froehle2014high}
Bradley Froehle and Per-Olof Persson.
\newblock A high-order discontinuous {G}alerkin method for fluid--structure
  interaction with efficient implicit--explicit time stepping.
\newblock {\em Journal of Computational Physics}, 272:455--470, 2014.

\bibitem{Dutt2000}
Alok Dutt, Leslie Greengard, and Vladimir Rokhlin.
\newblock Spectral deferred correction methods for ordinary differential
  equations.
\newblock {\em BIT Numerical Mathematics}, 40(2):241--266, 2000.

\bibitem{minion2003semi}
Michael~L. Minion.
\newblock Semi-implicit spectral deferred correction methods for ordinary
  differential equations.
\newblock {\em Communications in Mathematical Sciences}, 1(3):471--500, 2003.

\bibitem{bourlioux2003high}
Anne Bourlioux, Anita~T Layton, and Michael~L Minion.
\newblock High-order multi-implicit spectral deferred correction methods for
  problems of reactive flow.
\newblock {\em Journal of Computational Physics}, 189(2):651--675, 2003.

\bibitem{hagstrom2006on}
Thomas Hagstrom and Ruhai Zhou.
\newblock On the spectral deferred correction of splitting methods for initial
  value problems.
\newblock {\em Communications in Applied Mathematics and Computational
  Science}, 1(1):169--205, December 2006.

\bibitem{causley2019convergence}
Mathew Causley and David Seal.
\newblock On the convergence of spectral deferred correction methods.
\newblock {\em Communications in Applied Mathematics and Computational
  Science}, 14(1):33--64, 2019.

\bibitem{christlieb2014high}
Andrew Christlieb, Wei Guo, Maureen Morton, and Jing-Mei Qiu.
\newblock A high order time splitting method based on integral deferred
  correction for semi-lagrangian vlasov simulations.
\newblock {\em Journal of Computational Physics}, 267:7--27, 2014.

\bibitem{crockatt2017arbitrary}
Michael~M Crockatt, Andrew~J Christlieb, C~Kristopher Garrett, and Cory~D
  Hauck.
\newblock An arbitrary-order, fully implicit, hybrid kinetic solver for linear
  radiative transport using integral deferred correction.
\newblock {\em Journal of Computational Physics}, 346:212--241, 2017.

\bibitem{minion2018higher}
Michael~L Minion and RI~Saye.
\newblock Higher-order temporal integration for the incompressible
  navier--stokes equations in bounded domains.
\newblock {\em Journal of Computational Physics}, 375:797--822, 2018.

\bibitem{Pazner2016}
Will Pazner, Andrew Nonaka, John Bell, Marcus Day, and Michael Minion.
\newblock A high-order spectral deferred correction strategy for low {M}ach
  number flow with complex chemistry.
\newblock {\em Combustion Theory and Modeling}, 20(3):521--547, 2016.

\bibitem{Hairer1996}
Ernst Hairer and Gerhard Wanner.
\newblock {\em Solving ordinary differential equations {II}: Stiff and
  differential-algebraic problems}.
\newblock Springer Berlin Heidelberg, 1996.

\bibitem{Pazner2017}
Will Pazner and Per-Olof Persson.
\newblock Stage-parallel fully implicit {R}unge--{K}utta solvers for
  discontinuous {G}alerkin fluid simulations.
\newblock {\em Journal of Computational Physics}, 335:700 -- 717, 2017.

\bibitem{Hansen2011}
Anders~C. Hansen and John Strain.
\newblock On the order of deferred correction.
\newblock {\em Applied Numerical Mathematics}, 61(8):961--973, August 2011.

\bibitem{Tang2012}
Tao Tang, Hehu Xie, and Xiaobo Yin.
\newblock High-order convergence of spectral deferred correction methods on
  general quadrature nodes.
\newblock {\em Journal of Scientific Computing}, 56(1):1--13, October 2012.

\bibitem{bagnara1995unified}
Roberto Bagnara.
\newblock A unified proof for the convergence of jacobi and gauss--seidel
  methods.
\newblock {\em SIAM review}, 37(1):93--97, 1995.

\bibitem{tezduyar1986discontinuity}
TE~Tezduyar and YJ~Park.
\newblock Discontinuity-capturing finite element formulations for nonlinear
  convection-diffusion-reaction equations.
\newblock {\em Computer Methods in Applied Mechanics and Engineering},
  59(3):307--325, 1986.

\bibitem{estep2000estimating}
Donald~J Estep, Mats~G Larson, and Roy~D Williams.
\newblock {\em Estimating the error of numerical solutions of systems of
  reaction-diffusion equations}, volume 696.
\newblock American Mathematical Society, 2000.

\bibitem{estep2000using}
Donald~J Estep and Roland~W Freund.
\newblock Using {K}rylov-subspace iterations in discontinuous {G}alerkin
  methods for nonlinear reaction-diffusion systems.
\newblock In {\em Discontinuous Galerkin Methods}, pages 327--335. Springer,
  2000.

\bibitem{peraire2008compact}
Jaime Peraire and P-O Persson.
\newblock The compact discontinuous {G}alerkin ({CDG}) method for elliptic
  problems.
\newblock {\em SIAM Journal on Scientific Computing}, 30(4):1806--1824, 2008.

\bibitem{wall1999fluid}
Wolfgang~A Wall.
\newblock {\em Fluid-struktur-interaktion mit stabilisierten finiten
  elementen}.
\newblock Institut f{\"u}r Baustatik der Universit{\"a}t Stuttgart, 1999.

\bibitem{forster2007artificial}
Christiane F{\"o}rster, Wolfgang~A Wall, and Ekkehard Ramm.
\newblock Artificial added mass instabilities in sequential staggered coupling
  of nonlinear structures and incompressible viscous flows.
\newblock {\em Computer Methods in Applied Mechanics and Engineering},
  196(7):1278--1293, 2007.

\bibitem{gerbeau2003quasi}
Jean-Fr{\'e}d{\'e}ric Gerbeau and Marina Vidrascu.
\newblock A quasi-newton algorithm based on a reduced model for fluid-structure
  interaction problems in blood flows.
\newblock {\em ESAIM: Mathematical Modelling and Numerical Analysis},
  37(4):631--647, 2003.

\bibitem{kassiotis2011nonlinear}
Christophe Kassiotis, Adnan Ibrahimbegovic, Rainer Niekamp, and Hermann~G
  Matthies.
\newblock Nonlinear fluid--structure interaction problem. part i: implicit
  partitioned algorithm, nonlinear stability proof and validation examples.
\newblock {\em Computational Mechanics}, 47(3):305--323, 2011.

\bibitem{habchi2013partitioned}
Charbel Habchi, Serge Russeil, Daniel Bougeard, Jean-Luc Harion, Thierry
  Lemenand, Akram Ghanem, Dominique Della~Valle, and Hassan Peerhossaini.
\newblock Partitioned solver for strongly coupled fluid--structure interaction.
\newblock {\em Computers \& Fluids}, 71:306--319, 2013.

\bibitem{kuttler2006solution}
Ulrich K{\"u}ttler, Christiane F{\"o}rster, and Wolfgang~A Wall.
\newblock A solution for the incompressibility dilemma in partitioned
  fluid--structure interaction with pure dirichlet fluid domains.
\newblock {\em Computational Mechanics}, 38(4-5):417--429, 2006.

\bibitem{forster2006geometric}
Ch~F{\"o}rster, Wolfgang~A Wall, and Ekkehard Ramm.
\newblock On the geometric conservation law in transient flow calculations on
  deforming domains.
\newblock {\em International Journal for Numerical Methods in Fluids},
  50(12):1369--1379, 2006.

\bibitem{donea2003finite}
Jean Donea and Antonio Huerta.
\newblock {\em Finite element methods for flow problems}.
\newblock John Wiley \& Sons, 2003.

\bibitem{van2009added}
EH~Van~Brummelen.
\newblock Added mass effects of compressible and incompressible flows in
  fluid-structure interaction.
\newblock {\em Journal of Applied mechanics}, 76(2):021206, 2009.

\bibitem{de2012nonlinear}
Ren{\'e} De~Borst, Mike~A Crisfield, Joris~JC Remmers, and Clemens~V Verhoosel.
\newblock {\em Nonlinear finite element analysis of solids and structures}.
\newblock John Wiley \& Sons, 2012.

\bibitem{farhat1995mixed}
Charbel Farhat, Michel Lesoinne, and Nathan Maman.
\newblock Mixed explicit/implicit time integration of coupled aeroelastic
  problems: Three-field formulation, geometric conservation and distributed
  solution.
\newblock {\em International Journal for Numerical Methods in Fluids},
  21(10):807--835, 1995.

\bibitem{farhat1998torsional}
Ch~Farhat, C~Degand, B~Koobus, and M~Lesoinne.
\newblock Torsional springs for two-dimensional dynamic unstructured fluid
  meshes.
\newblock {\em Computer Methods in Applied Mechanics and Engineering},
  163(1-4):231--245, 1998.

\bibitem{farhat1998load}
Charbel Farhat, Michael Lesoinne, and Patrick Le~Tallec.
\newblock Load and motion transfer algorithms for fluid/structure interaction
  problems with non-matching discrete interfaces: Momentum and energy
  conservation, optimal discretization and application to aeroelasticity.
\newblock {\em Computer Methods in Applied Mechanics and Engineering},
  157(1-2):95--114, 1998.

\bibitem{chung1993time}
Jintai Chung and GM~Hulbert.
\newblock A time integration algorithm for structural dynamics with improved
  numerical dissipation: the generalized-$\alpha$ method.
\newblock {\em Journal of applied mechanics}, 60(2):371--375, 1993.

\bibitem{peng2009energy}
Zhangli Peng and Qiang Zhu.
\newblock Energy harvesting through flow-induced oscillations of a foil.
\newblock {\em Physics of Fluids}, 21(12):123602, 2009.

\bibitem{zahr2016adjoint}
Matthew~J Zahr and P-O Persson.
\newblock An adjoint method for a high-order discretization of deforming domain
  conservation laws for optimization of flow problems.
\newblock {\em Journal of Computational Physics}, 326:516--543, 2016.

\bibitem{persson2009discontinuous}
P-O Persson, J~Bonet, and J~Peraire.
\newblock Discontinuous {G}alerkin solution of the {N}avier--{S}tokes equations
  on deformable domains.
\newblock {\em Computer Methods in Applied Mechanics and Engineering},
  198(17-20):1585--1595, 2009.

\bibitem{roe1981approximate}
Philip~L Roe.
\newblock Approximate {R}iemann solvers, parameter vectors, and difference
  schemes.
\newblock {\em Journal of Computational Physics}, 43(2):357--372, 1981.

\bibitem{lin1995incompressible}
Chi-Kun Lin.
\newblock On the incompressible limit of the compressible {N}avier-{S}tokes
  equations.
\newblock {\em Communications in Partial Differential Equations},
  20(3-4):677--707, 1995.

\bibitem{desjardins1999incompressible}
Beno{\^i}t Desjardins, Emmanuel Grenier, P-L Lions, and Nader Masmoudi.
\newblock Incompressible limit for solutions of the isentropic
  {N}avier-{S}tokes equations with {D}irichlet boundary conditions.
\newblock {\em Journal de Math{\'e}matiques Pures et Appliqu{\'e}s},
  78(5):461--471, 1999.

\bibitem{Emmett2012}
Matthew Emmett and Michael Minion.
\newblock Toward an efficient parallel in time method for partial differential
  equations.
\newblock {\em Communications in Applied Mathematics and Computational
  Science}, 7(1):105--132, March 2012.

\end{thebibliography}
